\RequirePackage{fix-cm}
\documentclass[smallextended]{svjour3}       
\smartqed  
\usepackage{graphicx}
\usepackage{amssymb}
\usepackage{latexsym}
\usepackage{url}
\usepackage{xcolor}
\definecolor{newcolor}{rgb}{.8,.349,.1}
\usepackage{overpic}
\usepackage{subcaption}
\usepackage{amsmath}
\usepackage{bm}
\usepackage[numbers]{natbib}
\usepackage{tikz}
\usepackage[hidelinks]{hyperref}

\DeclareMathOperator{\diag}{diag}
\DeclareMathAlphabet\mathbfcal{OMS}{cmsy}{b}{n}
\journalname{SN PDEs with Applications}
\begin{document}

\title{Development and analysis of entropy stable no-slip wall
boundary conditions for 
the Eulerian model for viscous and heat 
conducting compressible flows}

\titlerunning{Solid wall boundary conditions for the Eulerian model}        

\author{Mohammed Sayyari
\and Lisandro Dalcin 
\and Matteo Parsani 
}


\institute{M. Sayyari \at
\email{mohammed.alsayyari@kaust.edu.sa}
\and L. Dalcin \at
\email{lisandro.dalcin@kaust.edu.sa}
\and M. Parsani \at
\email{matteo.parsani@kaust.edu.sa}\\
King Abdullah University of Science and Technology (KAUST), Computer Electrical 
and Mathematical Science and Engineering Division (CEMSE), Extreme Computing Research 
Center (ECRC), 23955-6900, Thuwal, Saudi Arabia
}

\date{Received: date / Accepted: date}

\maketitle


\begin{abstract}
    \ \\
    Nonlinear entropy stability analysis is used to derive 
    entropy stable no-slip wall boundary conditions for the Eulerian model 
    proposed by Sv\"{a}rd
    (\emph{Physica A: Statistical Mechanics and its Applications, 2018}). 
    and its spatial discretization based on entropy stable collocated discontinuous Galerkin 
    operators with the summation-by-parts property for unstructured grids.
    A set of viscous test cases of increasing complexity
    are simulated using both the Eulerian and the classic compressible Navier--Stokes 
    models. The
    numerical results obtained with the two models are 
    compared, and differences and similarities are then highlighted.
\end{abstract}


\section{Introduction}
The classical compressible Navier--Stokes (CNS) equations can be derived based
on the material (Lagrangian) derivative formulation \cite{schneiderbauer2013navier}.
In the Lagrangian sense, diffusion between gas pockets is non-existent, and thus, the continuity equation is 
hyperbolic. On the other hand, in the Eulerian model of \citet{svard2018new}, air molecules 
diffuse into other parts of the domain, and thus, the continuity equation 
is modeled as a parabolic equation.

Generally speaking, entropy conservation and entropy stability are used to preserve the 
second law of thermodynamics in the mathematical sense, \textit{i.e.}, the 
mathematical entropy function, $S$, decreases monotonically outside the 
equilibrium. This yields entropy estimates and bounds on $S$, which can 
be translated into bounds on the conservative variables, $q$, of 
the underlying model \cite{dafermos2005hyperbolic,svard2015weak}.
In this work, nonlinear entropy stability and the summation-by-parts (SBP) framework 
are used to derive entropy stable wall boundary conditions for the 
Eulerian model for the viscous and heat-conducting compressible flows proposed by 
\citet{svard2018new}
and their semidiscrete counterpart. 
As done in \cite{parsani2015entropy} 
for the classical CNS equations, a semidiscrete 
entropy estimate for the entire domain is achieved when the new boundary 
conditions are coupled with an entropy stable discrete interior operator. 
The data at the boundary are weakly imposed using a penalty flux approach 
and a simultaneous-approximation-term (SAT) penalty technique. At the semidiscrete
level, the work 
in \cite{parsani2015entropy} was sharpened in 
\cite{dalcin2019conservative} by first constructing entropy conservative wall boundary
conditions and then adding a precise interior penalty term. Here, we follow the semidiscrete analysis and derivation presented in
\cite{parsani2015entropy,dalcin2019conservative} by using a collocated
discontinuous Galerkin framework based on SBP operators constructed at the 
Legendre--Gauss--Lobatto (LGL) points. We verify the entropy conservation property of the baseline
boundary conditions by simulating the flow around a rotating sphere placed in a cubic domain.
In addition, to verify the accuracy of the proposed boundary
condition implementation, we present the convergence study for a three-dimensional 
test case constructed using the method of manufactured solutions.

In \citet{svard2018new,DOLEJSI2021110068}, numerical simulations were presented to highlight
the difference between the classical CNS equations and the Eulerian model.
Here, we also use a set of test cases of increasing complexity simulated using both the 
CNS and Eulerian models. We use the $hp$, fully-discrete 
entropy stable SSDC solver described, validated, and verified in \cite{parsani2020high}. 
The numerical results of the two models in regions near solid walls are compared, 
and differences and similarities are then highlighted.

The paper is organized as follows. In Section \ref{sec:ss-eul}, we present 
the Eulerian model in a general form, and we show its entropy stability analysis.
Then, in Section \ref{sec:wallbc}, we derive the wall boundary conditions at the continuous. 
Later, in Section \ref{sec:discrete}, we discretize the system using SBP-SAT operators 
and present a discrete entropy analysis of the Eulerian model. The 
latter step sets the context for the construction of discrete entropy 
conservative and entropy stable wall boundary conditions.
Furthermore, in Section \ref{sec:SAT-approach}, a common SAT procedure is 
presented to allow the use of a single subroutine for imposing interface coupling and
wall boundary conditions. In Section
\ref{sec:numerical}, we present a set of numerical results that demonstrate the
efficacy and accuracy of the new boundary conditions. In addition,
we present a few more test cases that highlight similarities and differences between
the two models slightly away from solid walls. Finally, conclusions are drawn in 
Section \ref{sec:conclusions}.


\section{Entropy analysis of the Eulerian model}\label{sec:ss-eul}

To analyze the Eulerian model, we begin by presenting its general form and then its
entropy analysis. The latter step will set the context for deriving solid wall 
boundary conditions that preserve the nonlinear entropy stability. 


\subsection{General form of the Eulerian model}

Sv\"{a}rd \cite{svard2018new} arrives at the following form of the Eulerian model for viscous and heat conducting 
compressible flows:
\begin{subequations}\label{eq:eulerian}
    \begin{align}
        \frac{\partial \rho}{\partial t}
        +\frac{\partial \rho \mathcal{U}_j}{\partial x_j}
        &=\frac{\partial}{\partial x_j}\left(\nu
        \frac{\partial \rho}{\partial x_j}\right),
        \label{eq:eulerian-general-continuity}\\
        \frac{\partial \rho \mathcal{U}_i}{\partial t}
        +\frac{\partial \rho \mathcal{U}_i \mathcal{U}_j}{\partial x_j}
        +\frac{\partial p}{\partial x_i}
        &=\frac{\partial}{\partial x_j}\left(\nu
        \frac{\partial \rho \mathcal{U}_i}{\partial x_j}\right),
        \label{eq:eulerian-general-momentum}\\
        \frac{\partial \rho \mathcal{E}}{\partial t}
        +\frac{\partial \rho \mathcal{E} \mathcal{U}_j}{\partial x_j}
        +\frac{\partial p \mathcal{U}_j}{\partial x_j}
        &=\frac{\partial}{\partial x_j}\left(\nu
        \frac{\partial \rho \mathcal{E}}{\partial x_j}\right),
        \label{eq:eulerian-general-energy}
    \end{align}
\end{subequations}
where $i,j=1:N_\text{DIM}$ (in MATLAB notation), and $\rho$, $\mathcal{U}_i$, $p$, and $\mathcal{E}$ are
the density, the velocity component in the $x_i$ direction, the thermodynamic pressure and the
specific total energy, respectively.

In what follows, we assume thermodynamically perfect 
(or ideal) gas. Thus, the thermodynamic pressure, $p$, is given by
\begin{equation}
    p=\rho R\mathcal{T},
\end{equation}
where $R$ is the gas constant. The specific total energy, $\mathcal{E}$, is given by the following equation
\begin{equation}
    \mathcal{E}=c_v\mathcal{T}+\frac12 \mathcal{U}_i\mathcal{U}_i,
\end{equation}
where $c_v$ is the heat capacity at constant volume and $T$ is the temperature.
The heat capacity at constant pressure, $c_p$, is related to $c_v$ through the gas constant, \textit{i.e.},
$R=c_p-c_v$. Furthermore, the generalized form of the kinematic viscosity, $\nu$, is given by
\begin{align}\label{eq:generalized-kinematic-viscosity}
    \nu=\frac{\alpha\mu}{\rho(\textbf{x},t)}+\beta(\rho,\mathcal{T}),
\end{align}
where $\mu$ is the dynamic viscosity, $\alpha\in[1,\frac43]$, 
and $\beta$ is an additional diffusion coefficient. Finally, the speed of sound, $c$, is defined as
\begin{equation}
    c=\sqrt{\gamma R\mathcal{T}},
\end{equation}
where $\gamma=\frac{c_p}{c_v}$.

\begin{remark}
    Equation \eqref{eq:generalized-kinematic-viscosity} is the most general expression of the kinematic viscosity of the 
	Eulerian model. However, in practice, \citet{svard2018new} chooses
    $\beta=0$ and thus, equation \eqref{eq:generalized-kinematic-viscosity} reduces to 
    a scaled kinematic viscosity coefficient
    \begin{align}
        \nu=\frac{\alpha\mu}{\rho(\textbf{x},t)}.
        \label{eq:scaled-kinematic-viscosity}
    \end{align}
\end{remark}


\subsection{Entropy analysis}

We cover the entropy analysis of the Eulerian model \eqref{eq:eulerian} 
by rewriting it using the following compact form,
\begin{align}\label{eq:conservative-form}
    \frac{\partial{q}}{\partial t}
    +\frac{\partial {f}^{(I)}_i}{\partial x_i}
    =\frac{\partial {f}^{(V)}_i}{\partial x_i},
\end{align}
where ${f}^{(I)}_i$, and ${f}^{(V)}_i$ are the inviscid and 
viscous fluxes in the $x_i$ direction, respectively.

The next theorem ensures that the Eulerian model is entropy stable.

\begin{theorem}\label{thm:entropy-eulerian}
    The following boundary integral
    \begin{align}\label{eq:entropy-bound}
        \int_\Gamma\left(
        {w}^\top \left(\nu\frac{dq}{dw}\right)
        \frac{\partial {w}}{\partial x_i}
        -\mathcal{F}_i\right)
        \ d\Gamma
    \end{align}
    bounds the time derivative of the entropy function, $S$, of the Eulerian model \eqref{eq:eulerian}.
\end{theorem}
\begin{proof}
    Given an entropy pair $(\mathcal{S},\mathcal{F}_i)$ 
    that symmetrizes the Eulerian model (see \cite{svard2018new}),
    we define the entropy variable 
    $w^\top=\frac{dS}{dq}$. Then, using the matrix
    $\frac{d{q}}{d{w}}$, we change the variables of
    system \eqref{eq:eulerian} to $q:=q(w)$. 
    These steps yield the symmetric form
    \begin{equation}
        \frac{d{q}}{d{w}}
        \frac{\partial {w}}{\partial t}
        +\frac{\partial {{g}}^{(I)}_i}{\partial x_i}
        =\frac{\partial {{g}}^{(V)}_i}{\partial x_i},
        \label{eq:eulerian-symmetric-form}
    \end{equation}
    where ${g}^{(I)}_i$, and ${g}^{(V)}_i$ are the symmetrized inviscid and 
    viscous fluxes, respectively. 
    Next, we multiply the symmetric system \eqref{eq:eulerian-symmetric-form} 
    from the left by the entropy variables. 
    Thus, using the chain rule, the term associated with the time derivative reads
    \begin{equation}\label{eq:dsdt}
        {w}^\top \frac{d{q}}{d{w}}
        \frac{\partial {w}}{\partial t}
        = \frac{d \mathcal{S}}{d{q}}
        \frac{d{q}}{d{w}}
        \frac{\partial {w}}{\partial t}
        =\frac{\partial \mathcal{S}}{\partial t}.
    \end{equation}
    Further, as shown by \citet{tadmor2003entropy}, the 
    contribution of the symmetrized inviscid flux term yields
    \begin{equation}\label{eq:inviscid-s-flux}
        {w}^\top \frac{\partial {{g}}^{(I)}_i}{\partial x_i}
        =\frac{\partial \mathcal{F}_i}{\partial x_i},
    \end{equation}
    \textit{i.e.}, the divergence of the entropy flux.
    Lastly, the viscous term contribution can be manipulated to obtain the following expression:
    \begin{equation}\label{eq:viscous-s-flux}
        w^\top\frac{\partial {{g}}^{(V)}_i}{\partial x_i}
        =w^\top\frac{\partial }{\partial x_i}
        \left(
            \left(\nu\frac{dq}{dw}\right)
            \frac{\partial {w}}{\partial x_i}
        \right).
    \end{equation}
    Therefore, using Equations \eqref{eq:dsdt}, \eqref{eq:inviscid-s-flux} and
    \eqref{eq:viscous-s-flux}, the scalar partial differential equation for the entropy function, 
    $\mathcal{S}$, integrated over a generic domain $\Omega$ with 
    boundary $\Gamma$ reads
    \begin{equation}
        \int_\Omega 
        \frac{\partial \mathcal{S}}{\partial t}
        +\frac{\partial \mathcal{F}_i}{\partial x_i}\ dx
        =\int_\Omega 
        w^\top\frac{\partial }{\partial x_i}
        \left(
            \left(\nu\frac{dq}{dw}\right)
            \frac{\partial {w}}{\partial x_i}
        \right)
        \ dx.
    \end{equation}
    Now, using the fundamental theorem of calculus on
    $\frac{\partial \mathcal{F}_i}{\partial x_i}$,
    and the integration-by-parts rule on 
    $w^\top\frac{\partial {{g}}^{(V)}_i}{\partial x_i}$ gives
    \begin{equation}
        \int_\Omega 
        \frac{\partial \mathcal{S}}{\partial t}
        =\int_\Gamma \left(
            w^\top
            \left(\nu\frac{dq}{dw}\right)
            \frac{\partial {w}}{\partial x_i}
            -\mathcal{F}_i
        \right)
        \ d\Gamma
        -DT,
    \end{equation}
    where
    \begin{equation}
        DT=\int_\Omega 
        \frac{\partial w^\top}{\partial x_i}
        \left(\nu\frac{dq}{dw}\right)
        \frac{\partial {w}}{\partial x_i}
        dx.
    \end{equation}
    Because the kinematic viscosity, $\nu$, is positive and the matrix 
    $\frac{dq}{dw}$ 
    is symmetric positive-definite, the product
    $\left(\nu\frac{dq}{dw}\right)$ is also a positive-definite matrix. 
    Thus, $DT$ is positive and appropriately bounded 
    (See \citet{svard2015weak}). We obtain the following expression
    \begin{equation}
    \label{eq:entropy-bound-eulerian}
        \int_\Omega 
        \frac{\partial \mathcal{S}}{\partial t}
        \ dx
            \leq \int_\Gamma \left(
            w^\top
            \left(\nu\frac{dq}{dw}\right)
            \frac{\partial {w}}{\partial x_i}
            -\mathcal{F}_i
        \right)
        \ d\Gamma
    \end{equation}
    which completes the proof.
\end{proof}

In the next section, we derive solid wall boundary conditions
such that the boundary integral on the left-hand side of 
\eqref{eq:entropy-bound-eulerian} is bounded by data.


\section{Entropy stable wall boundary condition}\label{sec:wallbc}

In this section, we derive
the boundary conditions for preserving the entropy stability 
of the system \eqref{eq:eulerian} for a solid wall. The analysis 
presented next closely follows the works of  
\cite{parsani2015entropy,dalcin2019conservative}.

The entropy-entropy flux pair $(\mathcal{S},\mathcal{F}_i)
=(-\rho s,-\rho \mathcal{U}_i s)$, used in \cite{dalcin2019conservative,parsani2015entropy}, 
is the only pair that admits a diffusive entropy
flux for the CNS model \cite{svard2018new}. Therefore, it is the only suitable
pair for a comparison between the CNS and Eulerian systems of equations.
This pair is also used in this work. 
Thus, the entropy 
variables are computed as \cite{fisher2012phd,carpenter2014entropy}
\begin{equation}
    w=\left(\frac{h}{\mathcal{T}}-s-\frac{\mathcal{U}_i\mathcal{U}_i}{2\mathcal{T}},
    -\frac{\mathcal{U}_1}{\mathcal{T}},
    -\frac{\mathcal{U}_2}{\mathcal{T}},
    -\frac{\mathcal{U}_3}{\mathcal{T}},
    -\frac{1}{\mathcal{T}}\right)^\top,
\end{equation}
where $h=c_{pr}\mathcal{T}$ is the enthalpy.

\begin{remark}
    The entropy-entropy flux pair for the Eulerian model is not necessarily restricted to 
    $(\mathcal{S},\mathcal{F}_i)=(-\rho s,-\rho \mathcal{U}_i s)$. 
    \citet{svard2018new} mentions that the Eulerian model admits a 
    diffusive entropy flux for all Harten's generalized entropies.
\end{remark}

To simplify, we rewrite the Eulerian form \eqref{eq:conservative-form} as
\begin{equation}
    \frac{\partial {q}}{\partial t}
    +\frac{\partial {f}^{(I)}_i}{\partial x_i}
    =\frac{\partial}{\partial x_i}
    (\nu\frac{\partial {q}}{\partial x_i}).
\end{equation}
From the proof of Theorem \ref{thm:entropy-eulerian}, 
we arrive at
\begin{equation}
    \frac{\partial \mathcal{S}}{\partial t}
    +\frac{\partial\mathcal{F}_i}{\partial x_i}
    ={w}^\top\frac{\partial}{\partial x_i}
    \left(\left(\nu\frac{dq}{dw}\right)
    \frac{\partial {w}}{\partial x_i}\right).
\end{equation}
Integrating the previous expression over a cubical domain,
$\Omega=(0,1)\times(0,1)\times(0,1)$,
and, without loss of generality, considering a solid wall located at $x_i=0$,
we get the following contributions
\cite{parsani2015entropy} 
\begin{equation}
    \int_\Omega \frac{\partial \mathcal{S}}{\partial t}\ dx
    =-\int_{x_i=0}\left[
    {w}^\top\left(\nu\frac{dq}{dw}\right)
    \frac{\partial {w}}{\partial x_i}-\mathcal{F}_i\right]
    dx_{j}dx_{k}
    -DT,
\end{equation}
where $j,k\neq i$. From Theorem \ref{thm:entropy-eulerian}, we get the bound
\begin{equation}\label{eq:entropy-bound-eulerian-at-a-wall}
    \int_\Omega 
    \frac{\partial \mathcal{S}}{\partial t}
    \ dx
    \leq -\int_{x_i=0}\left[
    {w}^\top\left(\nu\frac{dq}{dw}\right)
    \frac{\partial {w}}{\partial x_i}-\mathcal{F}_i\right]
    dx_{j}dx_{k}.
\end{equation}
To bound the time derivative of the entropy,
the right-hand-side of \eqref{eq:entropy-bound-eulerian-at-a-wall}
requires boundary data. For a solid viscous wall, assuming linear analysis,
five independent boundary conditions are required  \cite{svard2018new}. 
Four comes from the linear
analysis of the CNS model \cite{berg2011stable,svard2008stable},
in addition to a boundary condition that rises from
density diffusion \cite{svard2018new}. Four boundary conditions are the 
three no-slip boundary conditions, as the CNS case, and one on the
density flux presented below. The fifth condition
is the gradient of the temperature normal to the wall
(Neumann boundary condition; e.g., the adiabatic wall), or the 
temperature of the wall (the Dirichlet or isothermal wall boundary 
condition), or a mixture of Dirichlet and Neumann conditions 
(the Robin boundary condition). The last condition is the same
as for the CNS model \cite{berg2011stable,svard2008stable,svard2018new}. These five boundary conditions
lead to linear stability of both the CNS and Eulerian model \cite{berg2011stable,svard2008stable,svard2018new}.
In the remainder of this section, 
we will show the type and the form of the wall boundary conditions 
that have to be imposed to bound the estimate
\eqref{eq:entropy-bound-eulerian-at-a-wall} 
and, hence, to attain entropy stability.

\subsection{Inviscid contribution}

The inviscid contribution to the time rate of change of the entropy function,
\begin{equation}
    \int_{x_i=0}\mathcal{F}_i\ dx_{j}dx_{k}.
\end{equation}
appearing on the RHS of \eqref{eq:entropy-bound-eulerian-at-a-wall}, is treated as in
\cite[Theorem 2.2]{dalcin2019conservative}. However,
for clarity of presentation and completeness, we report it here.

\begin{theorem}\label{thm:no-slip-conditions}
    The no-slip boundary conditions $\mathcal{U}_i=0$ and $\mathcal{U}_j=\mathcal{U}_j^\text{wall}$,
    where $j\neq i$, 
	bound the inviscid contribution of the time derivative of the entropy function, \textit{i.e.}
    \begin{equation}
        \mathcal{F}_i=0,
    \end{equation}
    for an inviscid solid wall with an outfacing normal vector pointing in the $x_i$ direction.
\end{theorem}
\begin{proof}
  The proof of this theorem can be found in \cite[Theorem 2.2]{dalcin2019conservative}.
\end{proof}

\subsection{Viscous contribution}

In this subsection, we derive the viscous boundary conditions for a no-slip wall associated 
to the first term of the RHS of \eqref{eq:entropy-bound-eulerian-at-a-wall},
\textit{i.e.},
\begin{equation}
    -\int_{x_i=0}
    {w}^\top\left(\nu\frac{dq}{dw}\right)
    \frac{\partial {w}}{\partial x_i}\ 
    dx_{j}dx_{k}.
\end{equation}
The main result is provided in the following theorem.

\begin{theorem}\label{thm:wall-boundary-condition}
    The boundary conditions
    \begin{align}\label{eq:wall-boundary-conditions}
	    g(t)=\mu\frac{1}{\mathcal{T}}\frac{\partial\mathcal{T}}{\partial x_i}, 
        \qquad \textrm{and} \qquad
        \frac{1}{\rho}\frac{\partial \rho}{\partial x_i}=0,
    \end{align}
    bound the viscous contribution to the time derivative of 
    the entropy function for a no-slip solid wall with an outfacing 
    normal vector pointing in the $x_i$ direction.
\end{theorem}

\begin{proof}
    The explicit evaluation of the viscous contribution gives
    \begin{equation}\label{eq:wall-boundary-contribution}
        -{w}^\top\left(\nu\frac{dq}{dw}\right)
        \frac{\partial {w}}{\partial x_i}
        =\mu\frac{1}{\mathcal{T}}\frac{\partial\mathcal{T}}{\partial x_i}
        -\mu(s+(\gamma-1))\frac{1}{\rho}\frac{\partial \rho}{\partial x_i}
    \end{equation}
    that contributes positively to the time derivative of 
    the entropy function.  Then, setting $\frac{1}{\rho}\frac{\partial \rho}{\partial x_i}=0$
    yields the boundary condition
    \begin{equation}
        -{w}^\top\left(\nu\frac{dq}{dw}\right)
        \frac{\partial {w}}{\partial x_i}
        =\mu\frac{1}{\mathcal{T}}\frac{\partial T}{\partial x_i}=g(t)
    \end{equation}
    which completes the proof.
\end{proof}

For the CNS model, the viscous contribution to the time derivative of 
the entropy function can be bounded by only setting 
$\kappa\frac{1}{\mathcal{T}}\frac{\partial T}{\partial x_i}$
\cite{parsani2015entropy}. However, for the Eulerian model, we require 
the impostion of two boundary conditions, \textit{i.e.} 
$g(t) = \mu\frac{1}{\mathcal{T}}\frac{\partial\mathcal{T}}{\partial x_i}$ and 
$\frac{1}{\rho}\frac{\partial \rho}{\partial x_i}=0$.

In practice with the above boundary conditions, the time derivative of the entropy function
satisfies the following relation
\begin{equation}\label{eq:wall-data}
    \int_\Omega \frac{\partial \mathcal{S}}{\partial t}\ dx
    \leq \mu\frac{1}{\mathcal{T}}\frac{\partial\mathcal{T}}{\partial x_i}
    = g(t) = \text{DATA}.
\end{equation}
Bound \eqref{eq:wall-data} can be translated into a bound for the conserved
quantities \cite{dafermos2005hyperbolic,svard2015weak}, and hence, for the
primitive variables.


\section{Semidiscrete entropy stable framework}\label{sec:discrete}

Herein, using summation-by-parts (SBP) operators \cite{svard2014review} and the
simultaneous-approximation-technique (SAT) \cite{Carpenter1994,Nordstrom1999}, 
we provide an entropy stable 
framework of any order for the 
semidiscretization of the Eulerian model \eqref{eq:eulerian} using unstructured
grids.

\subsection{SBP operators}

The one-dimensional SBP operator for the first derivative in 
the direction $x_i$ is defined as the following.
\begin{definition}
    Summation--by--parts (SBP) operator for the first derivative: 
    A matrix operator with constant coefficients, 
    $\mathcal{D} \in \mathbb{R}^{N \times N}$, is a linear SBP operator 
    of degree $p$ approximating the derivative
    $\frac{\partial}{\partial x_i}$ on the domain \\
    $x_i \in \left[ a,b\right]$ with nodal distribution $x$ having $N$ nodes, if  
    \begin{enumerate}
        \item $\mathcal{D}x_i^j=jx_i^{j-1}$, $j = 0,1,\cdots,p; $
        \item $\mathcal{D} = \mathcal{P}^{-1}\mathcal{Q}$, where the norm matrix 
        $\mathcal{P}_{}$ is symmetric positive-definite;
        \item $\mathcal{Q}+\mathcal{Q}^{\top}=\mathcal{B}$, where 
        $\mathcal{B}=\mathrm{diag}\left[-1,0,\cdots,0,1\right]$.
    \end{enumerate}
    In other words, an SBP operator of degree $p$ is one that exactly 
    differentiates monomials up to degree $p$ ($p=N - 1$).
\end{definition}
In this work, a collocated discontinuous Galerkin approach is used. 
Specifically, diagonal norm SBP operators are constructed on the 
LGL nodes.
The one-dimensional SBP operators used in this work are explicitly constructed
in \cite{carpenter2015entropy}. Their extension to two- and three-dimensions
is achieved using tensor product operations 
\cite{carpenter2015entropy,parsani2015entropy}:
\begin{equation}\label{eq:SBP-tensor-matrices}
\begin{gathered}
    \mathcal{D}_{x_1}=\mathcal{D}_N\otimes \mathcal{I}_N \otimes \mathcal{I}_N \otimes \mathcal{I}_5, \quad \cdots \quad
    \mathcal{D}_{x_3}= \mathcal{I}_N \otimes \mathcal{I}_N \otimes \mathcal{D}_N\otimes \mathcal{I}_5, \\
    \mathcal{Q}_{x_1}=\mathcal{Q}_N\otimes \mathcal{I}_N \otimes \mathcal{I}_N \otimes \mathcal{I}_5, \quad \cdots \quad
	\mathcal{Q}_{x_3}= \mathcal{I}_N \otimes \mathcal{I}_N \otimes \mathcal{Q}_N\otimes \mathcal{I}_5,\\
    \mathcal{B}_{x_1}=\mathcal{B}_N\otimes \mathcal{I}_N \otimes \mathcal{I}_N \otimes \mathcal{I}_5, \quad \cdots \quad
    \mathcal{B}_{x_3}= \mathcal{I}_N \otimes \mathcal{I}_N \otimes \mathcal{B}_N\otimes \mathcal{I}_5,\\
    \Delta_{x_1}=\Delta_N\otimes \mathcal{I}_N \otimes \mathcal{I}_N \otimes \mathcal{I}_5,  \quad \cdots \quad
    \Delta_{x_3}= \mathcal{I}_N \otimes \mathcal{I}_N \otimes \Delta_N\otimes \mathcal{I}_5,\\
    \mathcal{P}_{x_1}=\mathcal{P}_N\otimes \mathcal{I}_N \otimes \mathcal{I}_N \otimes \mathcal{I}_5, \quad \cdots \quad 
    \mathcal{P}_{x_3}= \mathcal{I}_N \otimes \mathcal{I}_N \otimes \mathcal{P}_N\otimes \mathcal{I}_5,\\
    \mathcal{P}_{x_1,x_2}=\mathcal{P}_N\otimes \mathcal{P}_N \otimes \mathcal{I}_N \otimes \mathcal{I}_,5 \quad \cdots \quad
    \mathcal{P}_{x_2,x_3}= \mathcal{I}_N \otimes \mathcal{P}_N \otimes \mathcal{P}_N\otimes \mathcal{I}_5,\\
    \mathcal{P}_{x_1,x_2,x_3}= \mathcal{P}_N \otimes \mathcal{P}_N \otimes \mathcal{P}_N\otimes \mathcal{I}_5,
\end{gathered}
\end{equation}
where $\mathcal{D}_N$, $\mathcal{Q}_N$, $\mathcal{B}_N$, $\Delta_N$ and
$\mathcal{P}_N$ are the one-dimensional SBP operators, and $\mathcal{I}_N$ is the 
identity operator.\footnote{In this work, we use $N = 5$ because in three dimensions 
the number of partial differential equations is five for both the CNS and Eulerian model.}
In this context, we chose diagonal $\mathcal{P}_N$, 
and by the definition of SBP operators
$\mathcal{D}_N=\mathcal{P}_N^{-1}\mathcal{Q}_N$ and 
$\mathcal{B}_N=\mathcal{Q}_N^\top+\mathcal{Q}_N$.
The matrices $\mathcal{B}_{(\cdot)}$ pick off the interface terms in the respective directions. 
For the spectral element discretization considered in this paper, the $\mathcal{B}_{(\cdot)}$ matrices take 
on a particularly simple form; as an example, consider $\mathcal{B}_{x_1}$, which is given as
\begin{equation*}
\begin{gathered}
\mathcal{B}_{x_1} = \mathcal{B}_{x_{1}}^{+}-\mathcal{B}_{x_{1}}^{-},\\ 
\mathcal{B}_{x_1}^{-}=\diag\left(1,0,\dots,0\right)\otimes\mathcal{I}_{N}\otimes\mathcal{I}_{N}\otimes\mathcal{I}_{5},\\
\mathcal{B}_{x_1}^{+}=\diag\left(0,\dots,0,1\right)\otimes\mathcal{I}_{N}\otimes\mathcal{I}_{N}\otimes\mathcal{I}_{5}.
\end{gathered}
\end{equation*}
For a high-order accurate scheme on a tensor product cell, they pick off the
values of whatever vector they act on (typically the solution or the flux) at 
the nodes of the two opposite faces multiplied by the orthogonal component of 
the unit normal.

When applying any of these operators to the scalar entropy equation in space,
a hat will be used to differentiate the scalar operator from the full vector
operator, e.g.
\begin{equation*}
 \widehat{\mathcal{P}} = \left(\mathcal{P}_N \otimes \mathcal{P}_N \otimes \mathcal{P}_N \right).
\end{equation*}
We finally note that in the present work, the quadrature nodes and solution nodes are collocated.

\subsection{Semidiscretization of the Eulerian model}

Using the operators shown in \eqref{eq:SBP-tensor-matrices}, we can write the semidiscretization of 
\eqref{eq:eulerian} as
\begin{equation}\label{eq:eulerian-disc}
    \frac{\partial \mathbf{q}_i}{\partial t}=\mathcal{D}_{x_i}
    \left(\mathbf{f}^{(V)}_i-\mathbf{f}^{(I)}_i\right)
    +\mathcal{P}^{-1}_{x_i}\left(\mathbf{g}^{(b)}_{x_i}+\mathbf{g}^{(In)}_{x_i}\right),
\end{equation}
where vectors $\mathbf{g}^{(b)}_{x_i}$
enforce the boundary conditions, while $\mathbf{g}^{(In)}_{x_i}$ patches
interfaces together using a SAT approach \cite{parsani2015entropy}. The bolded 
letters represent quantities, $\mathbf{q}$, and functions, $\mathbf{f}$ and
$\mathbf{g}$ for all nodes in an element.

Following \cite{fisher2012phd,carpenter2014entropy}, we use the telescoping property of
an SBP operator,
\begin{equation}
    \mathcal{D}_{x_i}\mathbf{f}^{(I)}_i
    =\mathcal{P}^{-1}_{x_i}\mathcal{Q}_{x_i}\mathbf{f}^{(I)}_i
    =\mathcal{P}^{-1}_{x_i}\Delta_{x_i}\mathcal{I}_{StoF}\mathbf{f}^{(I)}_i
    =\mathcal{P}^{-1}_{x_i}\Delta_{x_i}\bar{\mathbf{f}}^{(I)}_i,
\end{equation}
to re-write the semidiscrete counterpart to the equation \eqref{eq:eulerian-disc}
as 
\begin{equation}\label{eq:eulerian-tele}
    \frac{\partial \mathbf{q}}{\partial t}=
    \left(\mathcal{D}_{x_i}\mathbf{f}^{(V)}_i
    -\mathcal{P}^{-1}_{x_i}\Delta_{x_i}\bar{\mathbf{f}}^{(I)}_i\right)
    +\mathcal{P}^{-1}_{x_i}\left(\mathbf{g}^{(b)}_{x_i}
    +\mathbf{g}^{(In)}_{x_i}\right).
\end{equation}
The vector $\bar{\mathbf{f}}^{(I)}_i$ is defined as 
$\bar{\mathbf{f}}^{(I)}_i=\mathcal{I}_{StoF}\mathbf{f}^{(I)}_i$ and the one-dimensional telescoping operator, $\Delta_{N}$, is given by
\begin{equation*}
\label{eq:delta}
 \Delta_N = \left(
 \begin{array}{cccccc}
  -1 & 1 & 0 & 0 & 0 & 0 \\
  0 & -1 & 1 & 0 & 0 & 0 \\
  0 & 0 & \ddots & \ddots & 0 & 0 \\
  0 & 0 & 0 & -1 & 1 & 0 \\
  0 & 0 & 0 & 0 & -1 & 1
 \end{array} 
 \right).
\end{equation*}
The operator $\mathcal{I}_{StoF}$ interpolates the value at the solution nodes to 
the interfaces between nodes as shown in Figure \ref{fig:lgl-points}.

Additionally, we define $\left[\frac{dq}{dw}\right]$
as a block diagonal matrix applied to all LGL points in an element.
The kinematic viscosity $\nu$ is a scalar and
we can bring it inside the matrix $\left[\frac{dq}{dw}\right]$, \textit{i.e.}, $\left[\nu\frac{dq}{dw}\right]$.
The block diagonal matrix $\left[\nu \frac{dq}{dw}\right]$ 
commutes with the operator $\mathcal{D}_{x_i}$. Thus, we can
re-write the viscous flux as 
\begin{equation}
    \mathbf{f}^{(V)}_i
    =\nu\mathcal{D}_{x_i}\mathbf{q}_i
    =\nu\mathcal{D}_{x_i}\left[\frac{dq}{dw}\right]\mathbf{w}_i
    =\left[\nu\frac{dq}{dw}\right]
    \mathcal{D}_{x_i}\mathbf{w}_i
    =\left[\nu\frac{dq}{dw}\right]
    \mathbf{\Theta},
\end{equation}
where the symbol $\mathbf{\Theta}_{x_i}$ is the gradient of the
entropy variables in the $x_i$ direction.


\begin{figure}[hbt!]
    \centering
    \includegraphics[width=.85\textwidth]{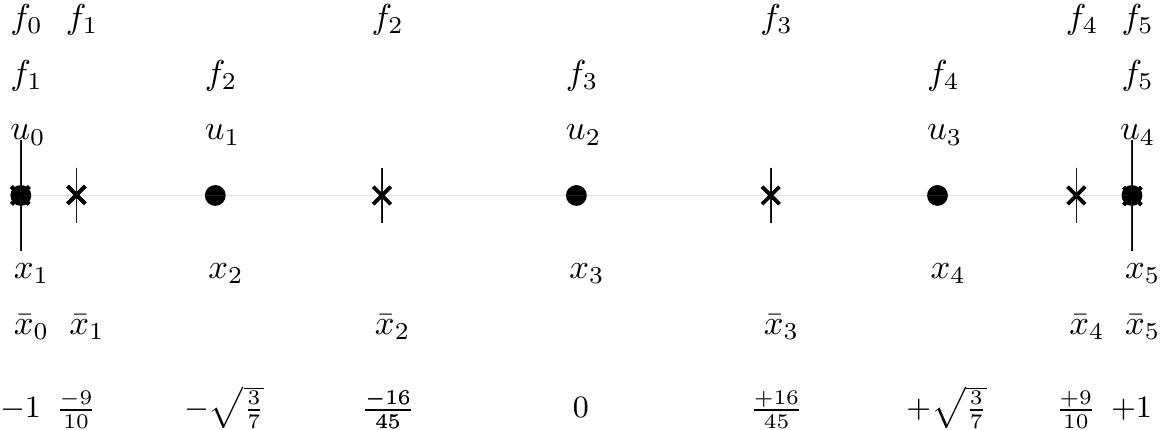}
    \caption{Onedimensional discretization using LGL points of order $p=4$. 
    $\cdot$ and $\times$ denote solution and flux points, respectively. 
    Reprinted from \cite{carpenter2014entropy} with permission.}
    \label{fig:lgl-points}
\end{figure}


Using the procedure based on local discontinuous Galerkin (LDG) and 
interior penalty approach (IP)
described in \cite{parsani2015entropy,dalcin2019conservative},
the semidiscretization \eqref{eq:eulerian-tele} can be recast as
\begin{subequations} \label{eq:eulerian-ldg-ip}
    \begin{align}
        &\frac{\partial \mathbf{q}_i}{\partial t}
        +\mathcal{P}^{-1}_{x_i}\Delta_{x_i}\bar{\mathbf{f}}^{(I)}_{i,sc}
	    -\mathcal{D}_{x_i}\left[\nu\frac{dq}{dw}\right]\mathbf{\Theta}_{x_i}
        =\mathcal{P}^{-1}_{x_i}
        \left(\mathbf{g}^{(b),q}_{x_i}
        +\mathbf{g}^{(In),q}_{x_i}\right),
        \label{eq:conservative-penalty}\\
        &\mathbf{\Theta}_{x_i}-\mathcal{D}_{x_i}\mathbf{w}
        =\mathcal{P}^{-1}_{x_i}\left(\mathbf{g}^{(b),\Theta}_{x_i}
        +\mathbf{g}^{(In),\Theta}_{x_i}\right).
        \label{eq:LDG-penalty}
    \end{align}
\end{subequations}
The terms
$\mathbf{g}^{(b),q}_{x_i}$, $\mathbf{g}^{(In),q}_{x_i}$, and
$\mathbf{g}^{(b),\Theta}_{x_i}$, $\mathbf{g}^{(In),\Theta}$
are the SAT penalty boundary ($b$) and
interface ($In$) terms on the conservative variables, $q$, and the gradient of the entropy
variables, $\Theta$, respectively.
 The contributions of the interface
penalty terms are non-zero only in the normal direction to the interface.

\begin{remark}
	To build a high-order entropy conservative (or, equivalently, entropy stable)
    semi discretization, the linear interpolation operator $\mathcal{I}_{StoF}$ is replaced
    with a nonlinear interpolation operator 
    \cite{fisher2012phd,carpenter2014entropy}.
    Thus, the operator $\mathcal{P}^{-1}_{x_i}\Delta_{x_i}\bar{\mathbf{f}}^{(I)}_{i,sc}$
    (or, $\mathcal{P}^{-1}_{x_i}\Delta_{x_i}\bar{\mathbf{f}}^{(I)}_{i,ss}$ for the entropy 
    stable case) is a nonlinear operator that is 
    a discrete counterpart to the term that 
    appears in Theorem \ref{thm:no-slip-conditions}. Thus, we arrive at the following relation
    \begin{equation}
        \mathbf{1}^\top\mathcal{P}\mathcal{P}^{-1}_{x_i}\Delta_{x_i}\bar{\mathbf{f}}^{(I)}_{i,sc}
        =\mathbf{1}^\top\widehat{\mathcal{P}}_{x_j,x_k}
        \widehat{\mathcal{B}}_{x_i}\mathbf{F}_{x_i}
        \simeq \int_{x_i=1}\mathcal{F}_i\ dx_jdx_k
        -\int_{x_i=0}\mathcal{F}_i\ dx_jdx_k.
    \end{equation}
    The flux vector $\bar{\mathbf{f}}^{(I)}_{i,sc}$
	is constructed using an entropy conservative two-point flux \cite{tadmor2003entropy}, and therefore,
    it satisfies the relation \eqref{eq:inviscid-s-flux}.
    In the case of an entropy stable flux, 
    $\bar{\mathbf{f}}^{(I)}_{i,ss}$, the resulting term is entropy
    dissipative \cite{carpenter2014entropy}.
\end{remark}


\subsection{Time derivative of entropy function}\label{sec:time-derivative-of-entropy}

Following the entropy analysis 
detailed in \cite{carpenter2014entropy,parsani2015entropy,parsani_entropy_stable_interfaces_2015}, 
the semidiscretization \eqref{eq:eulerian-ldg-ip} yields the following expression for the 
time derivative of the entropy function, $\mathcal{S}$:
\begin{align}
    \label{eq:time-derivative-of-entropy}
    &\frac{d}{d t}\mathbf{1}^\top\widehat{\mathcal{P}}\mathbfcal{S}
    +\mathbf{DT}=\mathbf{\Xi}
\end{align}
where
\begin{equation}\label{eq:DT}
    \mathbf{DT}
    =\left\|
    \left[\nu\frac{dq}{dw}\right]^{\frac12}
    \mathbf{\Theta}_{x_i}\right\|^2_\mathcal{P},
\end{equation}
and
\begin{align}
    \label{eq:Xi-general}
    \begin{split}
        \mathbf{\Xi}&=-\mathbf{1}^\top\widehat{\mathcal{P}}_{x_j,x_k}
        \widehat{\mathcal{B}}_{x_i}\mathbf{F}_{x_i}\\
        &+\mathbf{w}^\top\mathcal{P}_{x_j,x_k}\mathcal{B}_{x_i}
        \left[\nu\frac{dq}{dw}\right]
        \mathbf{\Theta}_{x_i}\\
        &+\mathbf{w}^\top\mathcal{P}_{x_j,x_k}
        \left(\mathbf{g}^{(b),q}_{x_i}
        +\mathbf{g}^{(In),q}_{x_i}\right)\\
        &+\left(
            \left[\nu\frac{dq}{dw}\right]
            \mathbf{\Theta}_{x_i}
        \right)^\top
        \mathcal{P}_{x_j,x_k}\left(\mathbf{g}^{(b),\Theta}_{x_i}
        +\mathbf{g}^{(In),\Theta}_{x_i}\right).
    \end{split}
\end{align}
The full derivation of \eqref{eq:time-derivative-of-entropy} is detailed
in Appendix \ref{sec:derivation-of-time-derivative-of-entropy}.

Now, consider a cubic element of length 1 with a solid wall with a normal vector in the $x_i$
direction,
we have that the interface penalty terms are zero, \textit{i.e.}, $\mathbf{g}^{(In),q}_{x_i}=\mathbf{g}^{(In),\Theta}_{x_i}=0$.
Therefore, equation \eqref{eq:Xi-general} reduces to
\begin{align}
    \label{eq:Xi}
    \begin{split}
        \mathbf{\Xi}&=-\mathbf{1}^\top\widehat{\mathcal{P}}_{x_j,x_k}
        \widehat{\mathcal{B}}_{x_i}\mathbf{F}_{x_i}\\
        &+\mathbf{w}^\top\mathcal{P}_{x_j,x_k}\mathcal{B}_{x_i}
        \left[\nu\frac{dq}{dw}\right]
        \mathbf{\Theta}_{x_i}\\
        &+\mathbf{w}^\top\mathcal{P}_{x_j,x_k}
        \mathbf{g}^{(b),q}_{x_i}\\
        &+\left(
            \left[\nu\frac{dq}{dw}\right]
            \mathbf{\Theta}_{x_i}
        \right)^\top
        \mathcal{P}_{x_j,x_k}\mathbf{g}^{(b),\Theta}_{x_i}.
    \end{split}
\end{align}


\subsection{Entropy stable wall boundary conditions for the 
semidiscrete system}
\label{subsec:ss_semi}

The boundary condition penalty term with respect to the conservative
variables is split into three design-order terms
plus one source boundary term:
\begin{equation}\label{eq:conservative-penalty-terms}
    \mathbf{g}^{(b),q}_{x_i}=
    \mathbf{g}^{(b,I),q}_{x_i}
    +\mathbf{g}^{(b,V),q}_{x_i}
    +\mathbfcal{M}^{(b,V)}
    +\mathbfcal{L}^{(b,V)}.
\end{equation}
The first component of each term is computed from the numerical solution, 
and the second component is constructed from a combination of the 
numerical solution and five independent boundary data,
as done in \cite{parsani2015entropy}.

The first term enforces Euler no-penetration wall
\begin{equation}\label{eq:no-penetration}
    \mathbf{g}^{(b,I),q}_{x_i}=-\frac12\mathcal{B}^-_{x_i}
    \left(\mathbf{f}^{(I)}_i
    -\mathbf{f}^{(b,I)}_{i,sc}(v^{(E)},v^{(b,I)})\right),
\end{equation}
where $\mathbf{f}^{(b,I)}_{i,sc}$ is the entropy conservative
flux and $v$ represent the primitive variables. The manufactured inviscid 
states are defined as in \cite{parsani2015entropy} for $i=1$, without loss of generality, 
\begin{align}\label{eq:inviscid-penalty-primitives}
    &v^{(E)}=\diag([1,-1,1,1,1]), \quad
    \textrm{and} \quad
    &&v^{(b,I)}=\left(\rho,-\mathcal{U}_1,\mathcal{U}_2,\mathcal{U}_3,\mathcal{T}\right)^\top.
\end{align} 

The viscous boundary term is given as
\begin{equation}\label{eq:neumann}
    \mathbf{g}^{(b,V),q}_{x_i}=\frac12\mathcal{B}^-_{x_i}
    \left(\left[\nu\frac{dq}{dw}\right]
    \mathbf{\Theta}_{x_i}-\mathbf{f}^{(b,V)}_i\right),
\end{equation}
where $\mathbf{f}^{(b,V)}_i$ will be defined 
later in this section.
Together with \eqref{eq:neumann}, the following term
\begin{equation}
    \mathbf{g}^{(b,V),\Theta}_{x_i}=\frac12\mathcal{B}^-_{x_i}
    \left(\mathbf{w}-\mathbf{w}^{(b,V)}\right),
\end{equation}
enforce the no-slip boundary condition weakly.
The analysis of the 
previous viscous terms match the analysis described by 
\cite{parsani2015entropy,dalcin2019conservative} since the 
entropy variables are the same,
and the matrix $\left[\nu\frac{dq}{dw}\right]$ is symmetric
positive-definite.

The remaining terms are defined in point-wise notation, and thus the bold notation is dropped.
The analysis of their contributions is
summarized next.
\begin{itemize}
    \item The manufactured viscous flux is defined as 
    \begin{equation}
        f^{(b,V)}_i
        =\left(\nu\frac{dq}{dw}\left(v^{(b,V)}\right)\right)
        \widetilde{\Theta}_{x_i},
    \end{equation}
    where the manufactured viscous boundary primitive variables, $v^{(b,V)}$,
    is defined, as in \cite{dalcin2019conservative}, by
    \begin{equation}\label{eq:viscous-boundary-primitive}
        v^{(b,V)}=\left(\rho,
        -\mathcal{U}_1+2\mathcal{U}_1^\text{wall},
        -\mathcal{U}_2+2\mathcal{U}_2^\text{wall},
        -\mathcal{U}_3+2\mathcal{U}_3^\text{wall},\mathcal{T}\right)^\top.
    \end{equation}
    The change of variables matrix, $\frac{dq}{dw}(v)$, as a function of
    the primitive variables, $v$, is defined in Appendix \ref{sec:dqdw}.
    The term $\widetilde{\Theta}_{x_i}$ is the manufactured gradient 
    of the entropy variables. The construction of $\widetilde{\Theta}_{x_i}$
    is summarized in Appendix \ref{sec:viscous-boundary-penalty}.
    \item The dissipative IP term is an averaged state
    of the viscous flux evaluated at the primitive state, $v$, and the 
    manufactured primitive state, $v^{(b,V)}$,
    \begin{equation}
        \mathcal{M}^{(b,V)}=-\beta\frac{f^{(V)}(v)+f^{(V)}(v^{(b,V)})}{2},
    \end{equation}
    where $\beta$ is a positive constant that is scaled with the inverse
    of the element length in the normal direction, controlling the strength 
    of the penalty term \cite{dalcin2019conservative}.
    \item The source term
    \begin{equation}
        \mathcal{L}^{(b,V)}=-\left(0,0,0,0,1\right)^\top g(t),
    \end{equation}
    imposes the heat flux boundary condition, as done in \cite{parsani2015entropy}.
\end{itemize}

We summarize the results for the RHS of \eqref{eq:time-derivative-of-entropy} in the
following three theorems. The first theorem is a result from \cite{parsani2015entropy} 
and is reprinted here for completeness

\begin{theorem}\label{thm:inviscid-boundary-penalty}
    The penalty inviscid flux contribution, $\mathbf{g}^{(b,I),q}$, 
    is entropy conservative if the vector $v^{(b,I)}$ is defined as in 
    \eqref{eq:inviscid-penalty-primitives}.
\end{theorem}

\begin{proof}
    The proof of this theorem can be found in \cite[Theorem 5.1]{parsani2015entropy}.
\end{proof}

\begin{theorem}\label{thm:viscous-boundary-penalty}
    The penalty terms for the viscous flux in the conserved 
    varaibles, $\mathbf{g}^{(b,V),q}$, and the gradient of entropy variables,
    $\mathbf{g}^{(b,V),\Theta}$,  are
    \begin{itemize}
        \item entropy conservative if the wall is adiabatic, \textit{i.e.} $g(t)=0$,
        \item entropy stable in the presence of a heatflux, $g(t)\neq 0$,
        where $g(t)$ is a given $L^2$ function.
    \end{itemize}
\end{theorem}

\begin{proof}
    Substituting the expressions for 
    $\mathbf{g}^{(b,V),q}$, and,
    $\mathbf{g}^{(b,V),\Theta}$ into \eqref{eq:Xi}, yields 
    \begin{equation}\label{eq:viscous-boundary-penalty}
        \frac{d}{d t}\mathbf{1}^\top\widehat{\mathcal{P}}\mathbfcal{S}
        +\mathbf{DT}
        =\mathbf{1}^\top\widehat{\mathcal{P}}_{x_j,x_k}g(t).
    \end{equation}
    The function $g(t)$ 
    for an adiabatic wall is $g(t)=0$. Thus the RHS is zero, and the scheme
    is entropy conservative. Otherwise, the term $g(t)$ is bounded by
    data. Hence, the RHS is bounded, completing the proof.
\end{proof}
The full details on the computation of \eqref{eq:viscous-boundary-penalty}
are reported in Appendix \ref{sec:viscous-boundary-penalty}.

\begin{theorem}\label{thm:dissipative-internal-penalty-term}
    The IP term, $\mathbfcal{M}^{(b,V)}$ , added to \eqref{eq:conservative-penalty-terms}
    is entropy dissipative.
\end{theorem}

\begin{proof}
    By expanding the penalty terms with respect to the conservative 
    variables, $g^{(b,V),q}$, and focusing only
    on the dissipative IP term, $\mathcal{M}^{(b,V)}$, in 
    \eqref{eq:time-derivative-of-entropy}, we arrive at
    \begin{equation}\label{eq:dissipative-internal-penalty-term}
        w^\top\mathcal{M}^{(b,V)}
        =-\frac{2\beta \alpha \mu}{R\mathcal{T}^2}\|\Delta \mathcal{U}\|^2
        \left(\|\Delta \mathcal{U}\|^2+R\mathcal{T}\right).
    \end{equation}
    Thus, $\mathcal{M}^{(b,V)}$ is entropy dissipative because all the 
    parameters and variables appearing on the RHS of \eqref{eq:dissipative-internal-penalty-term} 
    are positive. This completes the proof.
\end{proof}
The full details on the compution of \eqref{eq:dissipative-internal-penalty-term}
are reported in Appendix \ref{sec:dissipative-internal-penalty-term}.


\section{A common SAT procedure for the imposition of wall boundary conditions and interior 
interface coupling}\label{sec:SAT-approach}

The proposed approach for imposing the solid wall
boundary conditions allow for a SAT implementation that is identical
to the interface treatment shown in \cite{parsani_entropy_stable_interfaces_2015}.
We can use a single subroutine with different inputs corresponding to the imposition of
the interior interface couplings, or the adiabatic solid wall, or the wall with a prescribed heat entropy flow. In fact,
the interior interface coupling can be written as (see equations (16a-16d) in \cite{parsani_entropy_stable_interfaces_2015})
\begin{subequations}\label{eq:left-right-ldg-ip}
    \begin{align}
        \label{eq:left-ldg-ip-1}
        &\frac{\partial \mathbf{q}_l}{\partial t} 
        +\left(
            \mathcal{P}^{-1}_{x_{i,l}}\Delta_{x_{i,l}}\bar{\mathbf{f}}^{(I)}_{i,l,sc}
            -\mathcal{D}_{x_{i,l}} \left[\nu\frac{dq}{dw}\right] 
            \mathbf{\Theta}_{x_{i,l}}
        \right)
        =\mathcal{P}^{-1}_{x_{i,l}}\mathbf{g}_{x_{i,l}}^{(In),q},
        \\
        \label{eq:left-ldg-ip-2}
        &\mathbf{\Theta}_{x_{i,l}}
        -\mathcal{D}_{x_{i,l}}\mathbf{w} 
        =\mathcal{P}^{-1}_{x_{i,l}}\mathbf{{g}}_{x_{i,l}}^{(In),\Theta},
        \\
        \label{eq:right-ldg-ip-1}
        &\frac{\partial \mathbf{q}_r}{\partial t} 
        + \left(
            \mathcal{P}^{-1}_{x_{i,r}}
            \Delta_{x_{i,r}}
            \bar{\mathbf{f}}^{(I)}_{i,r,sc}
            -\mathcal{D}_{x_{i,r}}\left[\nu\frac{dq}{dw}\right] 
            \mathbf{\Theta}_{x_{i,r}}
        \right)
        =\mathcal{P}^{-1}_{x_{i,r}}\mathbf{{g}}_{x_{i,r}}^{(In),q},
        \\
        \label{eq:right-ldg-ip-2}
        &\mathbf{\Theta}_{x_{i,r}}
        -\mathcal{D}_{x_{i,r}}\mathbf{w}
        =\mathcal{P}^{-1}_{x_{i,r}}\mathbf{g}_{x_i,r}^{(In),\Theta}.
    \end{align}
\end{subequations}
The above equations have exactly the same structure as the LDG-IP approach used for the 
imposition of the solid wall boundary conditions except for the boundary 
penalty interface terms, $\mathbf{{g}}_{x_{i,r}}^{(b),\cdot}$ 
in equation \eqref{eq:eulerian-ldg-ip}, which are replaced by the interior penalty 
interface coupling terms, $\mathbf{{g}}_{x_{i,r}}^{(In),\cdot}$ in 
equations \eqref{eq:left-right-ldg-ip}.


\section{Numerical results}\label{sec:numerical}

In this section, we numerically investigate the proposed  entropy stable
wall boundary conditions. 
The numerical experiments reported in this paper are performed with 
the entropy stable collocated Discontinuous Galerkin algorithm and relaxation Runge--Kutta schemes implemented in the 
$hp$-adaptive, unstructured, curvilinear grid framework SSDC
\cite{parsani2020high}. SSDC is developed in the AANSLab, which is part of the Extreme Computing Research
Center at King Abdullah University of Science and Technology (KAUST).
The core entropy stable adaptive algorithms of SSDC are built on top of the Portable and Extensible Toolkit for Scientific
computing (PETSc) \cite{petsc-user-ref}, its mesh topology abstraction (DMPLEX)
\cite{knepley2009mesh}, and its scalable ODE/DAE solver library
\cite{petsc-ts}.
The SSDC framework uses a non-dimensional formalism; thus, all 
quantities are scaled to units. For the Eulerian model, the kinematic 
viscosity is scaled using $\alpha=1$. 
We use the two-point entropy conservative flux 
presented by \citet{chandrashekar2013kinetic}.
Furthermore, all the simulations have been
performed in double (machine) precision. For all the cases, we use the Runge--Kutta 
scheme of \citet{3bs} with adaptive time-step and both relative and absolute
tolerances set to $10^-8$. The meshes are generated using Gmsh 
\cite{geuzaine2009gmsh}, and Pointwise V18.3 released in September, 
2019. The SSDC Eulerian and SSDC CNS data computed for 
this section is available in \url{http://doi.org/10.5281/zenodo.5041436}.


\subsection{Convergence study}


\begin{figure}[htbp!]
    \centering
    \includegraphics[width=0.9\columnwidth,trim = 0 0 -5 0,clip]
    {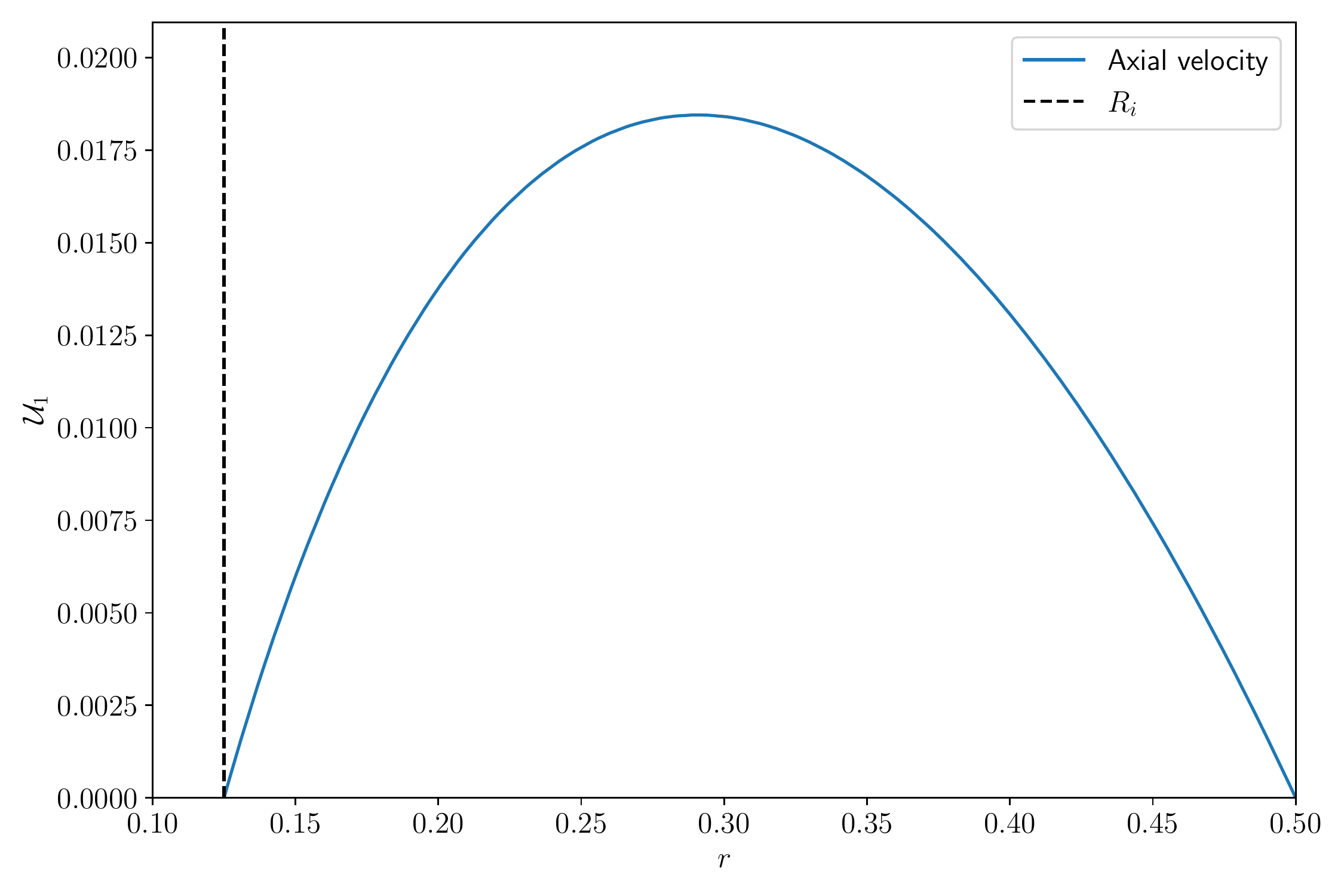}
    \caption{Axial velocity, $\mathcal{U}_1$, in a pipe with annular section; 
    Eulerian model with $\text{Re}=1$, $\text{Ma}=1e-3$.}
\end{figure}

\begin{table}[h!]
    \begin{center}
        \begin{tabular}{l|cccccc}
            Grid & $L^1$ & Rate & $L^2$ & Rate & $L^\infty$ & Rate\\
            \hline
            $4$ & $8.36276e-04$ & - & $1.77874e-03$ & - & $5.70764e-03$ & -\\
            $8$ & $1.25039e-04$ & -2.742 & $3.44481e-04$ & -2.368 & $1.49599e-03$ & -1.932\\
            $16$ & $1.99325e-05$ & -2.649 & $5.46416e-05$ & -2.656 & $3.12259e-04$ & -2.260\\
            $32$ & $2.64237e-06$ & -2.915 & $7.50755e-06$ & -2.864 & $5.07517e-05$ & -2.621\\
            $64$ & $3.31631e-07$ & -2.994 & $9.65089e-07$ & -2.960 & $7.22735e-06$ & -2.812\\
            $128$ & $4.12357e-08$ & -3.008 & $1.21565e-07$ & -2.989 & $9.60653e-07$ & -2.911
        \end{tabular}
        \caption{Convergence study for the flow in a pipe with annular cross-section at $\text{Re}=1$ and $\text{Ma}=1e-3$. 
        axial velocity error; solution polynomial degree: $p=2$.}
        \label{tab:plate-convergence-p2}
    \end{center}
\end{table}

\begin{table}[h!]
    \begin{center}
        \begin{tabular}{l|cccccc}
            Grid & $L^1$ & Rate & $L^2$ & Rate & $L^\infty$ & Rate\\
            \hline
            $4$ & $2.00203e-04$ & - & $3.85016e-04$ & - & $1.54031e-03$ & -\\
            $8$ & $2.71540e-05$ & -2.882 & $5.67563e-05$ & -2.762 & $3.16887e-04$ & -2.281\\
            $16$ & $2.15610e-06$ & -3.655 & $5.69093e-06$ & -3.318 & $4.42491e-05$ & -2.840\\
            $32$ & $1.24768e-07$ & -4.111 & $4.40714e-07$ & -3.691 & $4.63001e-06$ & -3.257\\
            $64$ & $6.57969e-09$ & -4.245 & $2.86670e-08$ & -3.942 & $4.05012e-07$ & -3.515\\
            $128$ & $2.95557e-10$ & -4.477 & $1.63672e-09$ & -4.131 & $3.12241e-08$ & -3.697
        \end{tabular}
        \caption{Convergence study for the flow in a pipe with annular cross-section at $\text{Re}=1$ and $\text{Ma}=1e-3$. 
        axial velocity error; solution polynomial degree: $p=3$.}
        \label{tab:plate-convergence-p3}
    \end{center}
\end{table}

\begin{table}[h!]
    \begin{center}
        \begin{tabular}{l|cccccc}
            Grid & $L^1$ & Rate & $L^2$ & Rate & $L^\infty$ & Rate\\
            \hline
              $4$ & $4.54267e-05$ & -      & $9.55211e-05$ & -      & $4.70479e-04$ & -\\
              $8$ & $4.47553e-06$ & -3.343 & $9.58936e-06$ & -3.316 & $6.31346e-05$ & -2.898\\
             $16$ & $1.88611e-07$ & -4.569 & $5.68142e-07$ & -4.077 & $5.14343e-06$ & -3.618\\
             $32$ & $5.97100e-09$ & -4.981 & $2.34227e-08$ & -4.600 & $2.75719e-07$ & -4.221\\
             $64$ & $1.85738e-10$ & -5.007 & $7.90368e-10$ & -4.889 & $1.13276e-08$ & -4.605\\
            $128$ & $5.76894e-12$ & -5.009 & $2.49773e-11$ & -4.984 & $3.99322e-10$ & -4.826
        \end{tabular}
        \caption{Convergence study for the flow in a pipe with annular cross-section at $\text{Re}=1$ and $\text{Ma}=1e-3$. 
        axial velocity error; solution polynomial degree: $p=4$.}
        \label{tab:plate-convergence-p4}
    \end{center}
\end{table}


In this section, we use the method of manufactured solutions (MMS) to verify the accuracy 
and correct implementation of the proposed boundary conditions.  We use a pipe with an annular section
and length $1$. Similarly to \cite{dalcin2019conservative},
this case is considered for two reasons.
First, it has an analytic solution for incompressible 
flow that cannot be represented in polynomial space \cite{rosenhead_bl_book}.
Second, it allows exercising the high-order mesh capabilities and better
approximate the circular geometry of the pipe.
The following solution is used for the axial velocity:
\begin{equation}\label{eq:mms-primitive}
    \mathcal{U}_1=\frac{G}{4\mu}\left((R_1^2-r^2)+(R_2^2-R_1^2)
    \frac{\log(r/R_1)}{\log(R_2/R_1)}\right),
\end{equation}
where $R_o=0.5$, $R_o/R_i=4$, and $G/\mu=1$. We consider uniform density and temperature,
and zero nonaxial velocities.
The parameters used are $\text{Re}=1$, $\text{Ma}=1e-3$ and $\alpha=1$.
No-slip adiabatic wall boundary conditions are used on the outer and inner cylinder whereas,
periodic boundary conditions are used on the remaining boundaries.
We use Mathematica to compute the source term of the Eulerian model \cite{Mathematica}.

The error calculation uses the following discrete norms
\begin{equation}
    \begin{split}
        &\text{Discrete }L^{1}: \|\mathbfcal{U}\|_{L^{1}}
        =\|\Omega\|^{-1}\sum\limits_{k=1}^{K}\mathbf{1}_{N_{k}}^\top
        \mathcal{P}_k\mathbfcal{J}_k\text{abs}\left(\mathbfcal{U}_k\right),\\
        &\text{Discrete }L^{2}: \|\mathbfcal{U}\|_{L^{2}}^2
        =\|\Omega\|^{-1}\sum\limits_{k=1}^{K}\mathbfcal{U}_k^\top
        \mathcal{P}_k\mathbfcal{J}_k\mathbfcal{U}_k,\\
        &\text{Discrete }L^{\infty}: \|\mathbfcal{U}\|_{L^{\infty}}
        =\max\limits_{k=1\dots K}\text{abs}\left(\mathbfcal{U}_k\right),
    \end{split}
\end{equation}
where $\|\Omega\|$ represents the volume of the computational domain, 
$\mathbfcal{J}_k$ is the metric Jacobian of the curvilinear transformation from physical
space to computational space of the $k$-th hexahedral element, and $K$ is the
total number of non-overlapping hexahedral elements in the mesh.

The results of the convergence study are shown in Tables \ref{tab:plate-convergence-p2}, \ref{tab:plate-convergence-p3}
and \ref{tab:plate-convergence-p4}, where the first column reprisents
the number of elements in the radial and angular coordinates. We observe that 
the computed order of accuracy is approximately $(p+1)$.


\subsection{Spinning sphere: verification of entropy conservation}


\begin{figure}[htbp!]
    \begin{subfigure}{.45\textwidth}
        \centering
        \includegraphics[width=\columnwidth,trim = 0 0 -5 0,clip]
            {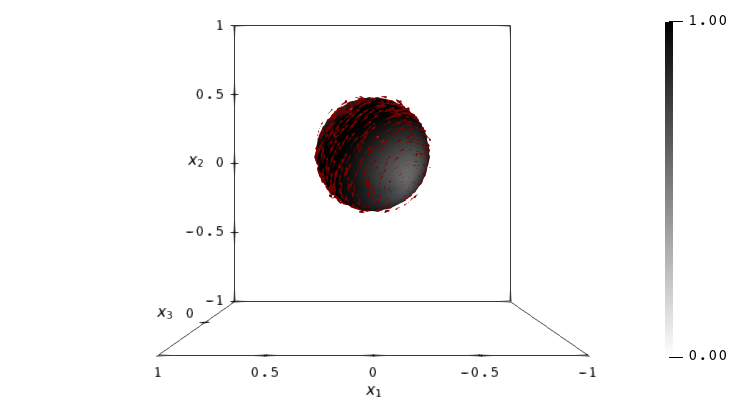}
        \caption{Overview of the spinning sphere.}
        \label{fig:sphere-full}
    \end{subfigure}
    \begin{subfigure}{.45\textwidth}
        \centering
        \includegraphics[width=\columnwidth,trim = 0 0 -5 0,clip]
            {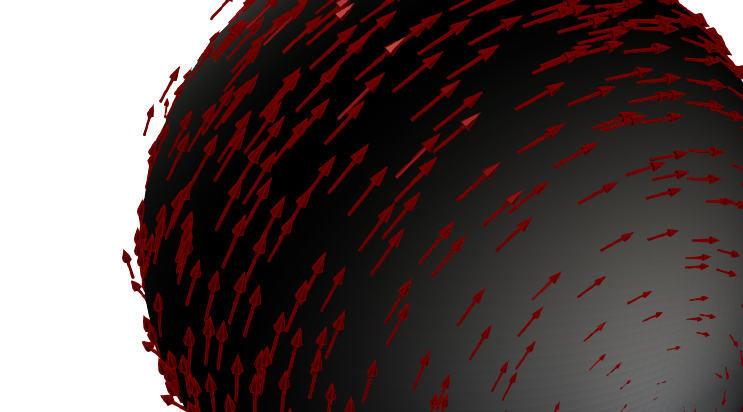}
        \caption{A closer look at the velocity field near the spinning
        sphere.}
        \label{fig:spinning}
    \end{subfigure}
    \caption{Spinning sphere enclosed in a cubical box; 
        Eulerian model with $\text{Re}=1$, and $\text{Ma}=0.05$. 
        The magnitude of the pointwise velocity is used to scale the arrow glyphs
      vectors and to color the surface of the sphere.}
        \label{fig:spinning-sphere}
\end{figure}

\begin{figure}[htbp!]
\begin{subfigure}{.49\textwidth}
    \centering
    \includegraphics[width=0.975\columnwidth,trim = 0 0 -5 0,clip]
        {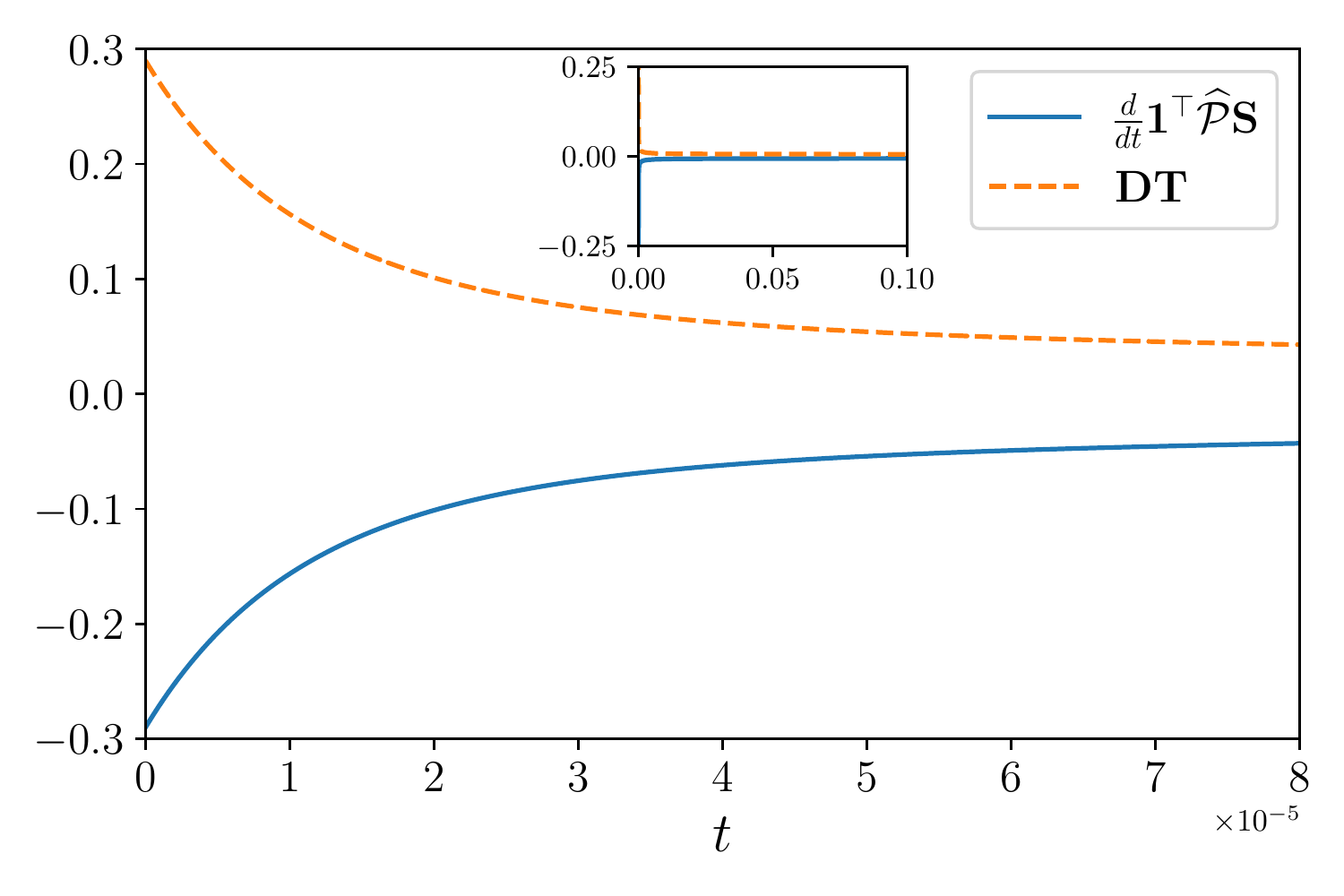}
        \vspace{0.25cm}
    \caption{Evolution of the time derivative of the entropy function and
    the dissipation term.}
    \label{fig:dsdt-DT}
\end{subfigure}
\begin{subfigure}{.49\textwidth}
    \centering
    \includegraphics[width=\columnwidth,trim = 0 0 -5 0,clip]
        {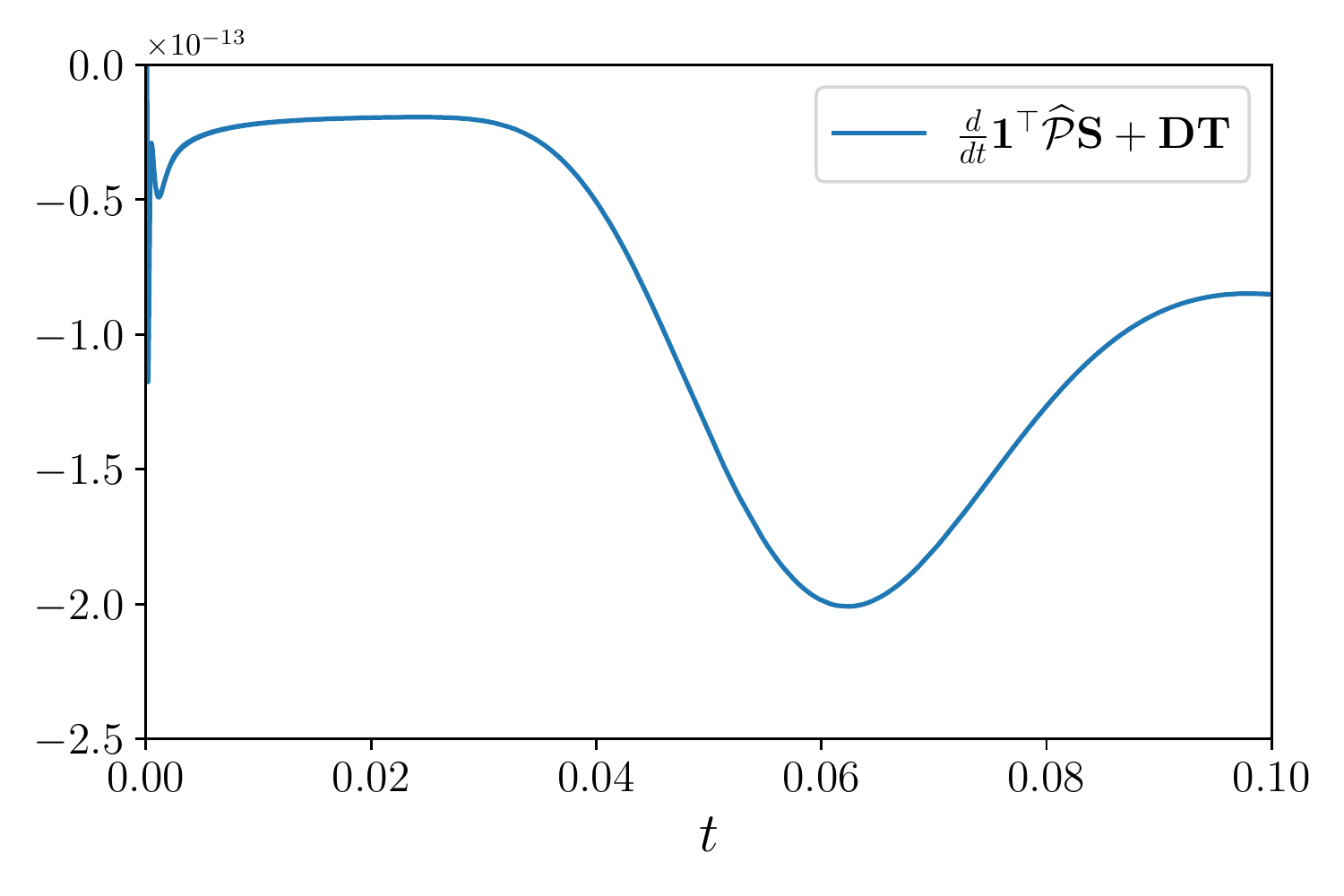}
    \caption{Evolution of the sum of the time derivative of the entropy function and
    the dissipation term.}
    \label{fig:dsdtpDT}
\end{subfigure}
\caption{Verification of the entropy conservation for the spinning sphere enclosed in a cubical box; 
    Eulerian model with $\text{Re}=1$, and $\text{Ma}=0.05$.}
\end{figure}


To verify the entropy conservation property of the new boundary conditions
without IP term, 
we simulate a spinning sphere enclosed in a cubic domain. The
sphere rotates at a constant angular velocity around a unit vector given by
$\hat{a_r}=\frac{a_r}{\|a_r\|}$, where $a_r=(1,1,1)^\top$. 
The domain size is $2\times 2\times 2$ with a sphere of diameter $D=
0.6$ located at the center of it. Solid wall boundary conditions are imposed on
the sphere surface and all six faces of the cubic box.
The mesh is composed of $4,374$ hexahedral elements.
The sphere surface is discretized with quadratic boundary elements. A solution
polynomial degree $p=5$ is used.
We run the Eulerian model with $\text{Re}=1$, and $\text{Ma}=0.05$.
In Figure \ref{fig:spinning-sphere}, the velocity vector field near the surface
of the sphere is shown. The sphere is also colored based on the module of its
pointwise velocity vector.

Figure \ref{fig:dsdt-DT} shows the time
derivative of the entropy function of the semidiscretization of the Eulerian model
\eqref{eq:eulerian-disc},
$\frac{d}{d t}\mathbf{1}^\top\widehat{\mathcal{P}}\mathbfcal{S}$,
and the  dissipation term, $\mathbf{DT}$,  
as a function of time, $t$ (see Equation \eqref{eq:time-derivative-of-entropy}). 
As shown in Figure \ref{fig:dsdtpDT}, the two terms cancel out, 
up to machine precision. Therefore, the procedure proposed to impose the 
boundary conditions is entropy conservative if the IP term in \eqref{eq:conservative-penalty-terms}  
is set to zero. This simulation is a strong numerical verification of what is proven in
Section \ref{subsec:ss_semi}.


\subsection{Laminar flow around a cylinder $\text{Re}=40$}


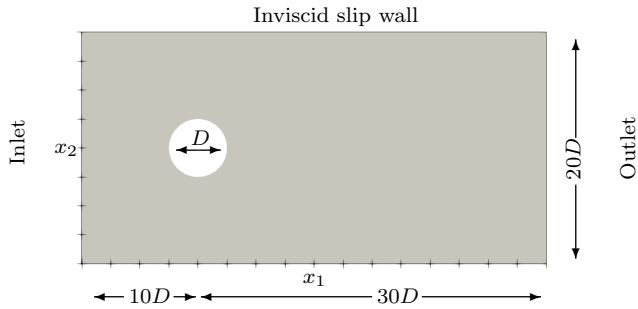
\begin{figure}[htbp!]
    \centering
    \begin{overpic}[width=0.55\textwidth,tics=5]{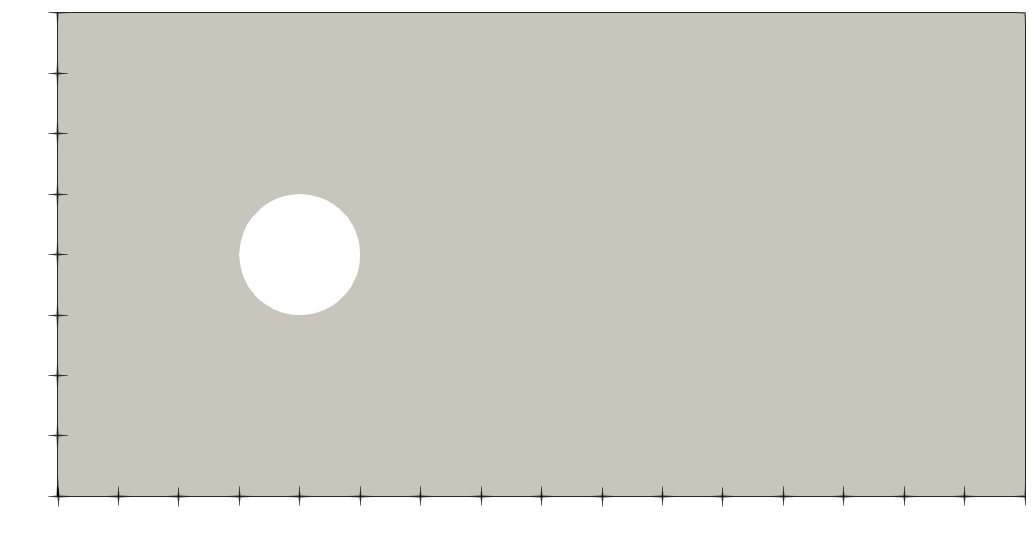}
        \put (50,1) {$x_1$}
        \put (0,27.5) {$x_2$}
        \put (40,54) {\rotatebox{0}{Inviscid slip wall}}

        \linethickness{0.5pt}
        \put (-9,25) {\rotatebox{90}{Inlet}}
        \put (114,22) {\rotatebox{90}{Outlet}}
	    \put (27.5,29) {$D$}
        \put(29,28){\color{black}\vector(1,0){4.5}}
        \put(29,28){\color{black}\vector(-1,0){4.5}}
	    \put (65,-3) {\rotatebox{0}{$30D$}}
        \put(74,-1.75){\color{black}\vector(1,0){24}}
        \put(64,-1.75){\color{black}\vector(-1,0){34.5}}
	    \put (15,-3) {\rotatebox{0}{$10D$}}
        \put(24,-1.75){\color{black}\vector(1,0){4.5}}
        \put(14,-1.75){\color{black}\vector(-1,0){6}}
	    \put (103,22) {\rotatebox{90}{$20D$}}
        \put(105,30.5){\color{black}\vector(0,1){20}}
        \put(105,21){\color{black}\vector(0,-1){15}}

    \end{overpic}
    \vspace{0.3cm}
    \caption{Cylinder domain with $D=1$.}
    \label{fig:vis-cyl-domain}
\end{figure}

\begin{figure}[]
    \centering
    \definecolor{cnscolor}{HTML}{7a0000}
    \begin{overpic}[width=0.95\textwidth,tics=5]{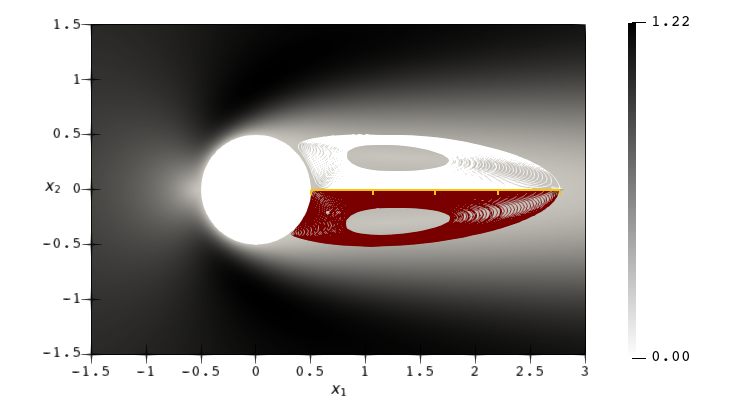}
        \put (48.5,33.5) {\rotatebox{0}{\small\color{black}Eulerian}}
        \put (51,24.75) {\rotatebox{0}{\small\color{cnscolor}CNS}}
    \end{overpic}
    \caption{Wake region of the flow past a cylinder at $\text{Re}=40$ and $\text{Ma}=0.07$. 
    The background is colored by the magnitude of the pointwise
	velocity vector.}
    \label{fig:vis-wake-40}
\end{figure}

\begin{table}
    \begin{center}
        \begin{tabular}{l|ll}
            & $L_\text{bubble}$ & $C_D$\\
            \hline\\
            \citet{park1998numerical} & $2.25$ & $1.51$\\
            \citet{ye1999accurate} & $2.27$ & $1.52$\\
            \citet{fornberg1980numerical} & $2.24$ & $2.5$\\
            \citet{dennis1970numerical} & $2.35$ & $1.52$\\
            \citet{calhoun2002cartesian} & $2.18$ & $1.62$\\
            \citet{russell2003cartesian} & $2.29$ & $1.6$\\
            SSDC CNS & $2.254$ & $1.608$\\
            SSDC Eulerian & $2.256$ & $1.609$
        \end{tabular}
    \end{center}
    \caption{Bubble length, $L_\text{bubble}$, and drag coefficient,
      $C_D$.}
    \label{tab:wak-cd-Re-40}
\end{table}

\begin{figure}[h]
    \centering
    \centering
    \begin{subfigure}{0.49\textwidth}\centering 
        \includegraphics[width=1.00\columnwidth,trim = 0 0 -5 0,clip]
        {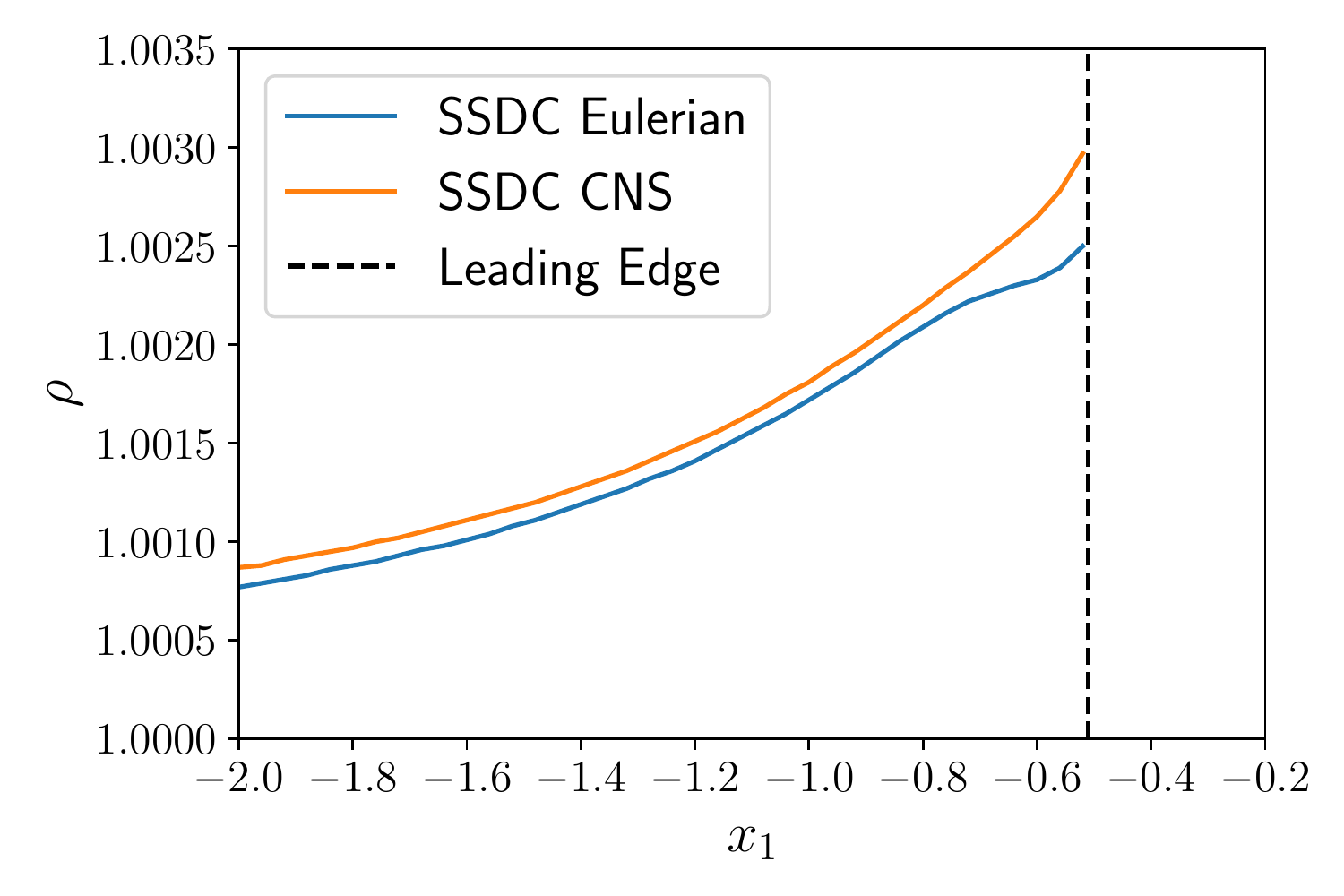}
        \caption{Leading edge.}
        \label{subfig:cyl_leading}
    \end{subfigure}
    \vspace{0.5cm}
    \begin{subfigure}{0.49\textwidth}\centering 
        \includegraphics[width=1.00\columnwidth,trim = 0 0 -5 0,clip]
        {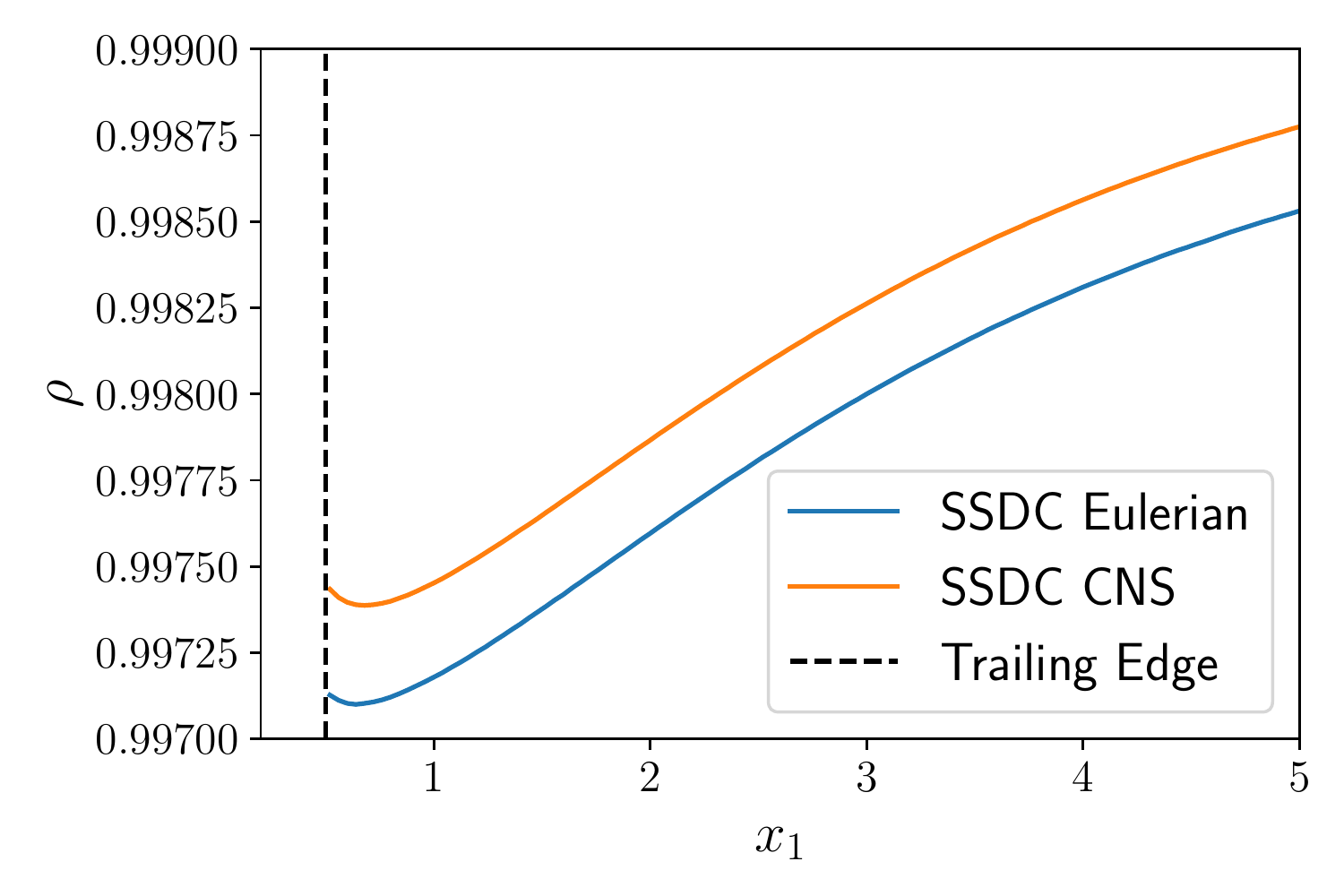}
        \caption{Trailing edge.}
\label{subfig:cyl_trailing}
    \end{subfigure}
    \caption{The density profile for the flow past a cylinder at $\text{Re}=40$ and
    $\text{Ma}=0.07$.}
    \label{fig:vis-cyl-40-mean-density}
\end{figure}

\begin{figure}[]
    \centering
    \begin{overpic}[width=\columnwidth,trim = 0 0 -5 0,clip]
        {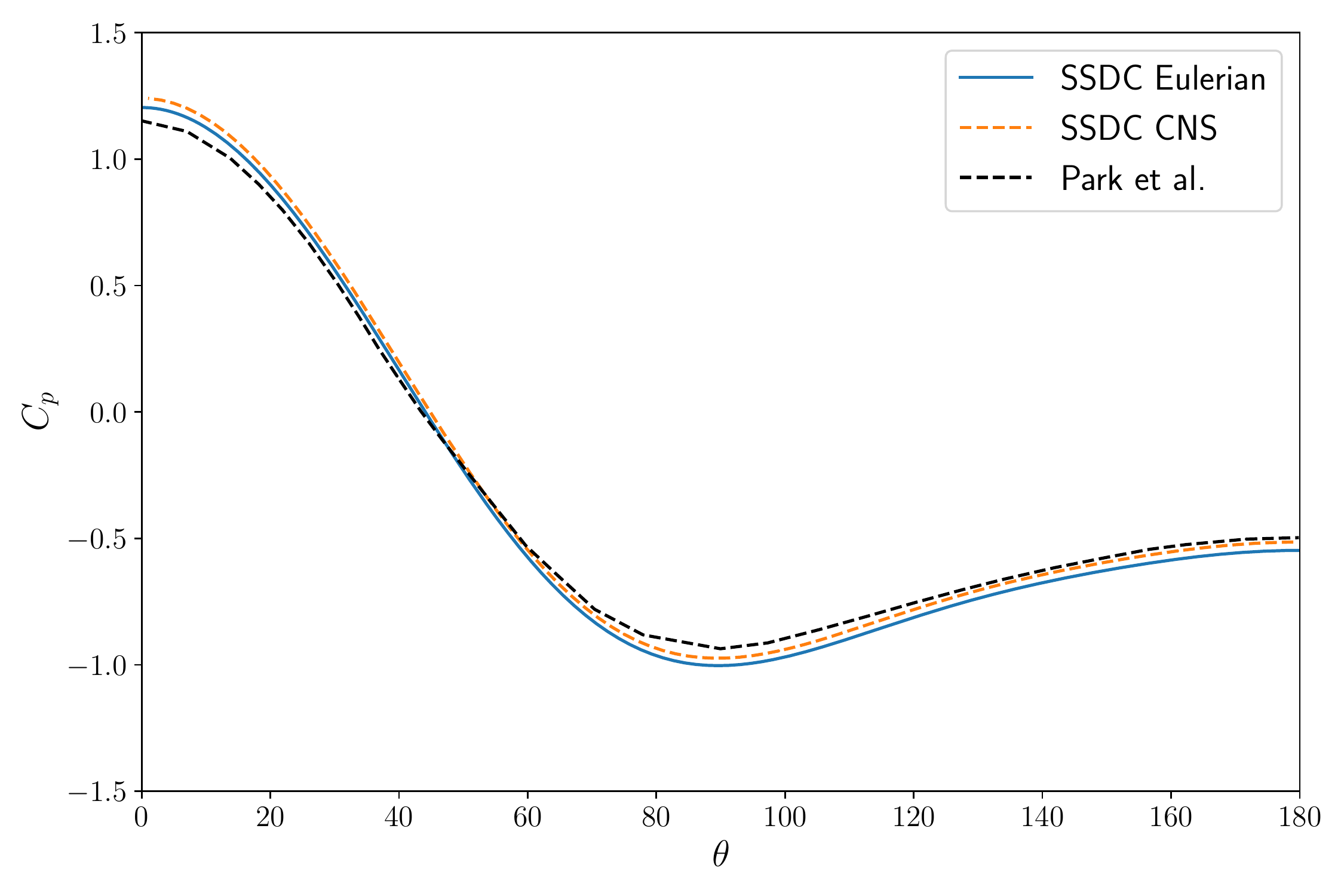}
        \put (89.75,52.25) {\cite{park1998numerical}}
    \end{overpic}        
    \caption{Wall pressure coefficient for the flow past a cylinder at $\text{Re}=40$
    and $\text{Ma} = 0.07$. The reference data
    are obtained from \cite{park1998numerical}.}
    \label{fig:vis-cyl-40-cp}
\end{figure}

\begin{figure}[]
    \centering
    \begin{overpic}[width=\columnwidth,trim = 0 0 -5 0,clip]
        {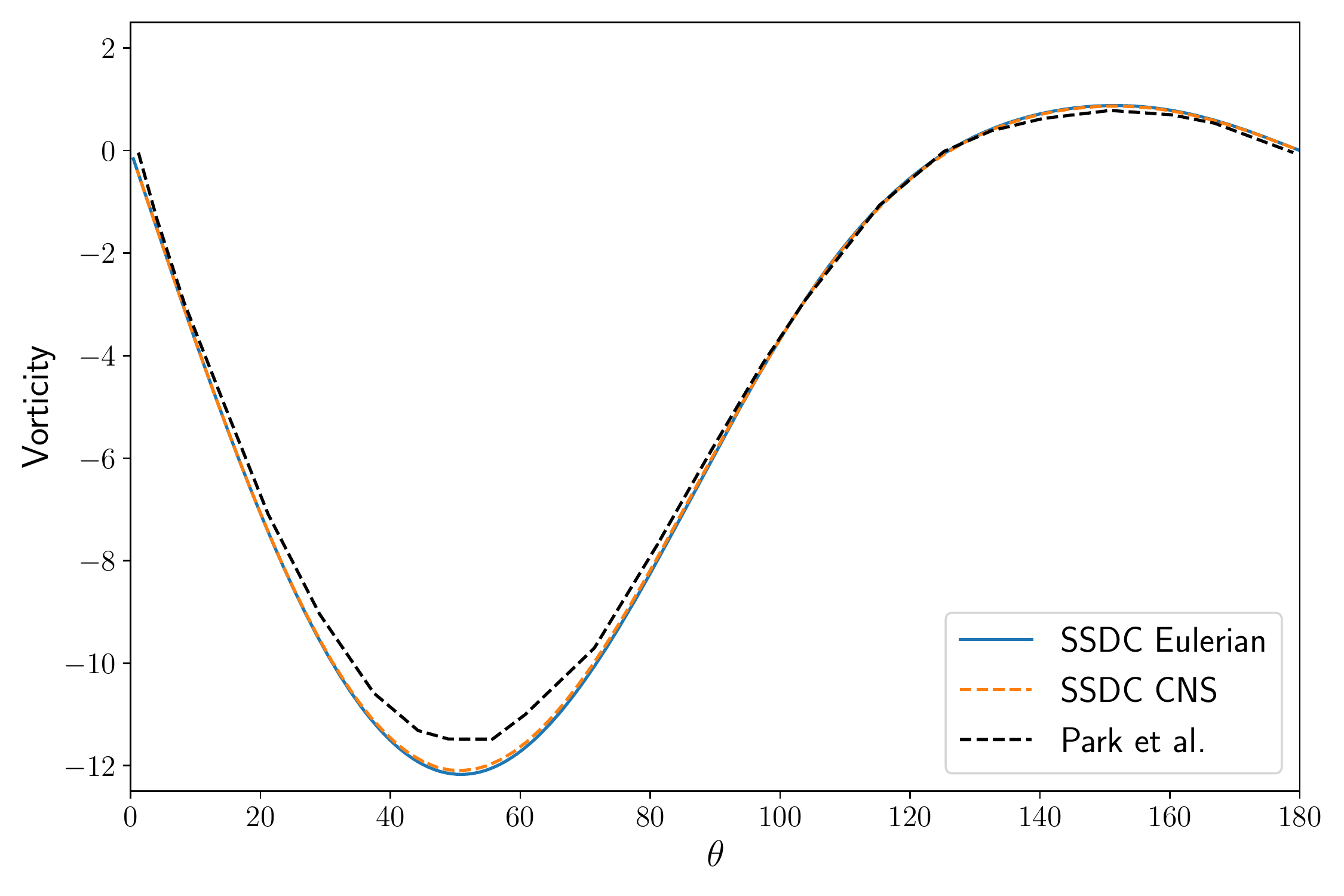}
        \put (89.75,10.75) {\cite{park1998numerical}}
    \end{overpic}
    \caption{Vorticity at the wal for the flow past a cylinder at $\text{Re}=40$ and $\text{Ma}=0.07$. 
    The reference data
    was obtained from \cite{park1998numerical}.}
    \label{fig:vis-cyl-40-vort}
\end{figure}


In this section, we present the two-dimensional flow around a cylinder of diameter $D = 1$ at
$\text{Re}=40$, and $\text{Ma}=0.07$. 
Figure \ref{fig:vis-cyl-domain} shows the computational 
domain and summarizes the boundary conditions.
The initial condition is a uniform flow in the $x_1$ direction.
The case is run until convergence to a steady-state. Specifically, we stop the simulation when
the residual reaches $10^{-14}$ and plateau around that value. 
Entropy stable adiabatic no-slip wall boundary conditions are enforced on the cylinder surface.
Far-field boundary conditions are used for the remaining boundaries.
The mesh is
composed of $1,140$ elements with quadratic edges for representing the surface of the
cylinder. We use a solution polynomial degree $p =7$. Therefore, the total
number of degrees of freedom is $72,960$. The solution computed with this setup is
denoted here as a ``numerically converged solution" in the sense that if we increase the order of
accuracy of the method, $p$, or refine the mesh, the difference between two consecutive numerical solutions for all the
primitive variables is machine precision.

First, we compute the length of the recirculation bubble in the wake region. 
Figure 
\ref{fig:vis-wake-40} shows the streamlines of the circulation bubble downstream
of the sphere. The background is the magnitude of the pointwise velocity vector. The top part of the plot
corresponds to the solution computed with the Eulerian model, whereas the bottom
part is the solution obtained using the CNS model. Qualitatively, the two
solutions look identical. Furthermore, both models show a wake length of 
$2.25$.  The first column in Table \ref{tab:wak-cd-Re-40} compares the length of the 
bubble with several results published in the literature. The agreement is good.
The second column of Table \ref{tab:wak-cd-Re-40} compares the drag coefficient,
$C_D$. The values obtained with both models are very close to each other and agree well with
that reported by \citet{calhoun2002cartesian} and \citet{russell2003cartesian}. 

Next, we present some measurements in the boundary layer of the cylinder wall. 
Figure \ref{fig:vis-cyl-40-mean-density} shows
the density profile for both models, where for the Eulerian model, we enforce a
zero normal density gradient at the wall. Furthermore, for both models, we use
adiabatic solid wall boundary conditions. Because in our simulations, we use a 
non-dimensional formalism, the difference between the two solutions is not negligible, 
especially near the leading edge; see Figure \ref{subfig:cyl_leading}.
The difference observed is accumulated at the leading edge and remains consistent along
with the tailwind, as shown in Figure \ref{subfig:cyl_trailing}.

Figure \ref{fig:vis-cyl-40-cp} shows the wall pressure coefficient, $C_p$.
The curves obtained with the Eulerian model and the CNS model are very close to each other and
follow relatively well the DNS results published by \citet{park1998numerical}.
A similar comparison is made in Figure \ref{fig:vis-cyl-40-vort} for the vorticity at the wall. 
Again, the curves obtained with the Eulerian model and the CNS model are very close to each other
and follow well the DNS results except around $50^{\circ}$ where they both undershoot the 
reference results.


\subsection{Blast-wave}

\begin{figure}
    \centering
    \includegraphics[scale=0.5]{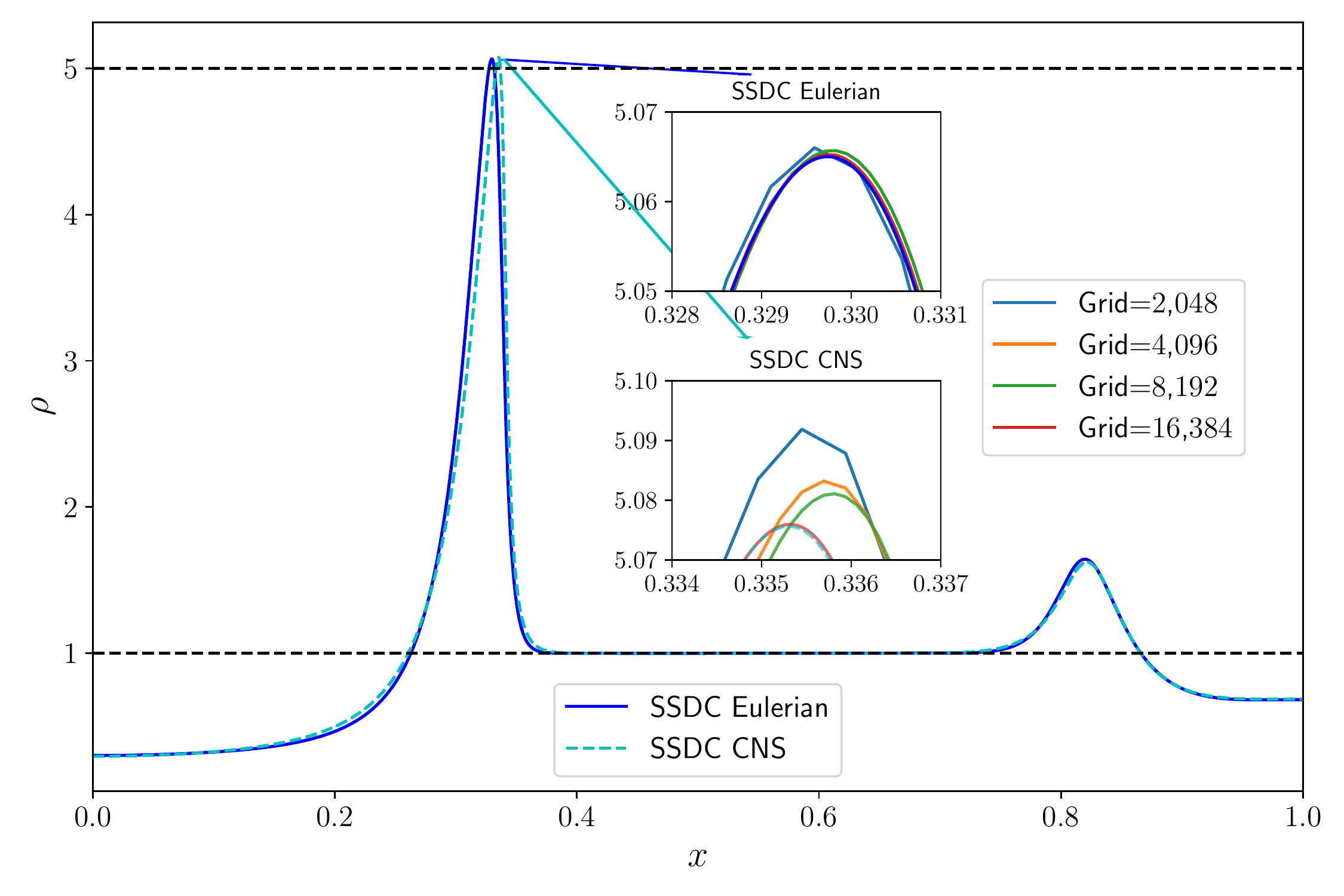}
    \caption{Density distribution for the blastwave at time $t=0.01$ computed
    with the Eulerian and the CNS models. The base figure uses a grid of $32,768$ 
    elements.}
    \label{fig:blastwave-density-comparison}
\end{figure}


\citet{svard2018new} presents the one-dimensional blastwave as an example where the 
implementation of the CNS model converges to an unphysical solution.
Here, we investigate the same test case. We use $\text{Re}=10$ and $\text{Ma}=0.07$, and 
the same initial conditions as in \cite{svard2018new}.
All the simulations are performed with $p=1$. We present the results for
a sequence of five uniform nested grids, with respectively $2,048$, $4,096$, 
$8,192$, $16,384$, and $32,768$ elements. The solutions obtained with the finest grid are
denoted here as ``numerically converged solutions". Solid wall adiabatic no-slip boundary
conditions with zero density gradient for the Eulerian model are imposed at the
left and right boundaries of the domain.

\citet{svard2018new} observes that since the density flux is
not bounded, the solution can converge to a state where the wall dissipates
density. This is a non-physical solution and is computed with a
non-entropy stable implementation of the CNS model 
\cite{svard2018new}. Our discretization of the Eulerian and 
CNS models are entropy stable.
As we will see shortly, we are unable to 
reproduce the ``lagging" effect observed in 
\cite{svard2018new}. 

The main plot in Figure \ref{fig:blastwave-density-comparison} compares the 
density profiles computed with the Eulerian and the CNS models on the finest grid with 
$8,192$ elements. We can easily observe that
the two curves do not lie on top of each other. Specifically, we notice
a clear difference in the region near the maximum peak, whose position along the 
$x$ axis also appears to be influenced by the model. Furthermore, differently
from \cite{svard2018new}, we do not notice any difference at the left and right 
boundaries of the domain between the solution computed with the Eulerian and the
CNS models.
In Figure \ref{fig:blastwave-density-comparison}, we also
report two zoom
views of the maximum peaks. There, we plot the density profiles
obtained with the five nested grids. We can observe how the
solutions approach the ``numerically converged solution". 
In particular, it appears that the solution obtained with the Eulerian
model reaches the ``converged state" faster than the solution obtained with the CNS model.
Furthermore, we notice that the value of the maximum peak obtained with the Eulerian model is slightly smaller
than the same value computed by the CNS model.


\subsection{Supersonic flow around a cylinder}


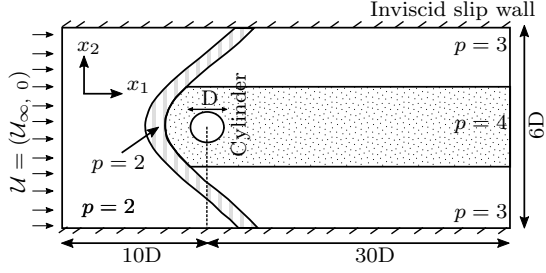
\begin{figure}[]
    \centering
    \begin{overpic}[width=0.55\textwidth,tics=5]{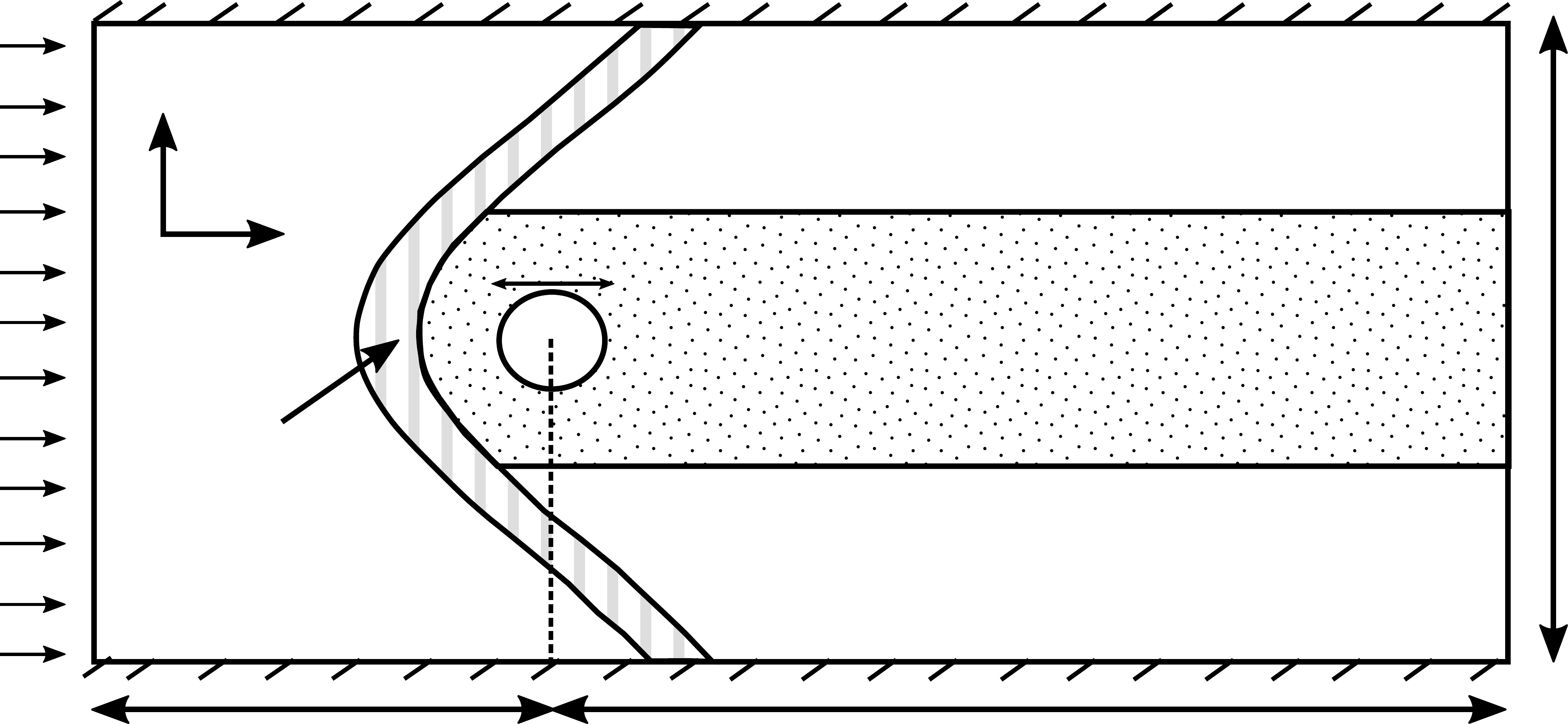}
        \put (19,32) {$x_1$}
        \put (9,40) {$x_2$}
        \put (40,18.5) {\rotatebox{90}{\small{Cylinder}}}
        \put (34,29) {$\mathrm{D}$}
        \put (-4.75,11.25) {\rotatebox{90}{${\mathcal{U}}=\left(\mathcal{U}_{\infty,~0}\right)$}}
        \put (100,22) {\rotatebox{90}{$6\mathrm{D}$}}
        \put (68,47) {\rotatebox{0}{Inviscid slip wall}}
        
        \put (65,-3) {\rotatebox{0}{$30\mathrm{D}$}}
        \put (18,-3) {\rotatebox{0}{$10\mathrm{D}$}}
        
        \put (85,24) {$p=4$}
        \put (85,40) {$p=3$}
        \put (85,6) {$p=3$}
        \put (10,7) {$p=2$}
        \put (10,7) {$p=2$}
        \put (12,16) {$p=2$}
    \end{overpic}
    \vspace{0.3cm}
    \caption{Computational domain, $hp-$nonconforming grid, and solution polynomial degree 
    distribution, $p$, for the simulation of the supersonic flow past a circular cylinder.
    The underlying figure was adapted from \cite{parsani2020high}.}
    \label{fig:domain-p-super-cyl}
\end{figure}

\begin{figure}[]
    \centering
    \begin{subfigure}{0.85\textwidth}\centering 
        \includegraphics[width=1.00\columnwidth,trim = 0 0 -5 0,clip]
        {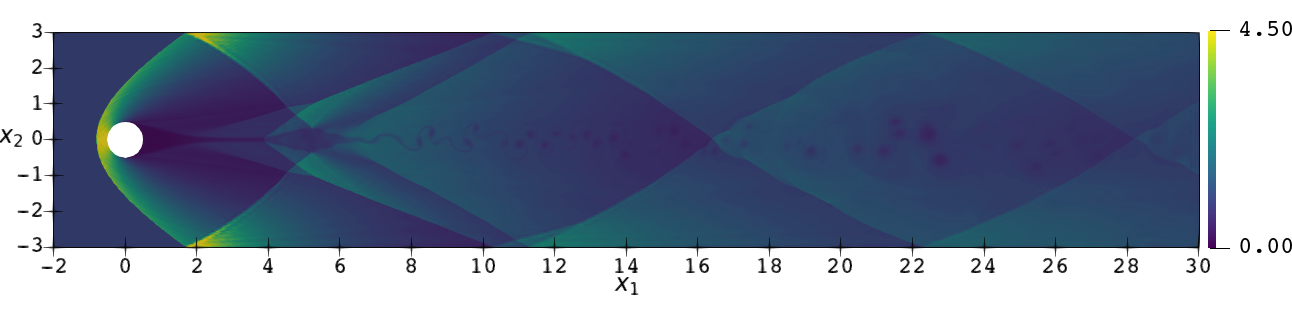}
        \caption{Density, $\rho$.}
    \end{subfigure}\\
    \vspace{0.5cm}
    \begin{subfigure}{0.85\textwidth}\centering 
        \includegraphics[width=1.00\columnwidth,trim = 0 0 -5 0,clip]
        {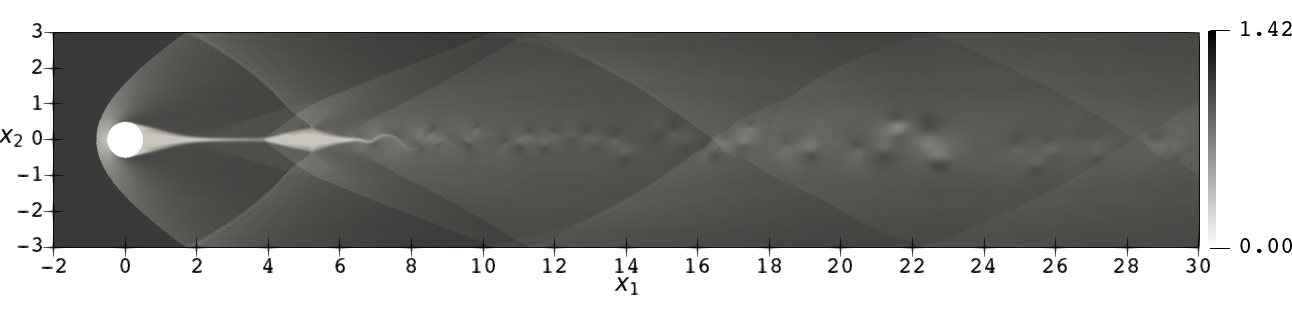}
        \caption{Velocity component in the $x_{1}$ direction, $\mathcal{U}_1$.}
        \label{fig:sup-cyl-velocity-x}
            \end{subfigure}
    \caption{Supersonic flow past a cylinder enclosed between two plates at $\text{Re}=10,000$ and $\text{Ma} = 3.5$.}
    \label{fig:sup-cyl-fields}
\end{figure}

\begin{figure}[]
    \centering
    \includegraphics[width=0.95\columnwidth,trim = 0 0 -5 0,clip]
            {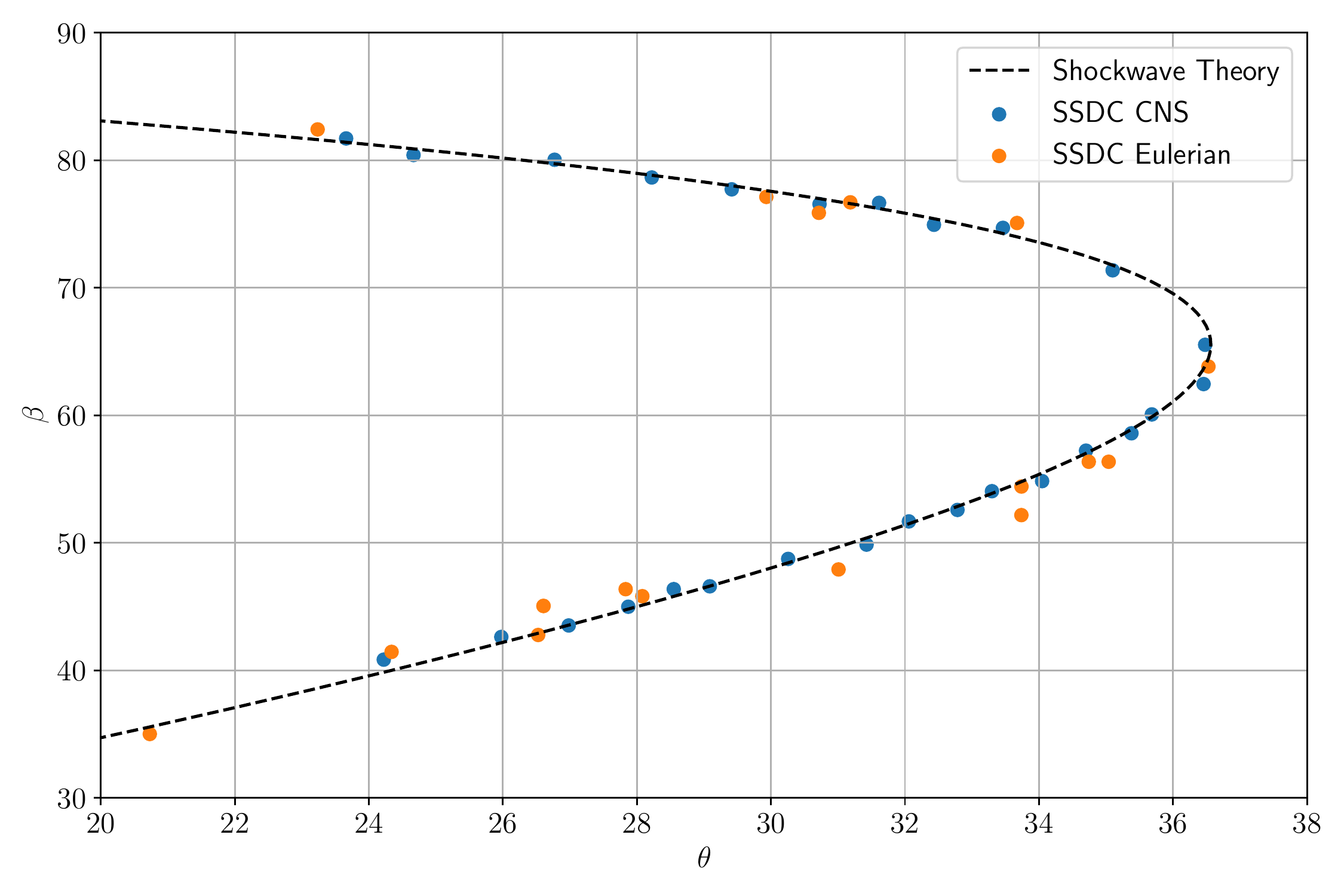}
    \caption{$\theta$-$\beta-\text{Ma}$ plot for the flow past a cylinder at $\text{Re}=10,000$,
	and $\text{Ma} = 3.5$.}
    \label{fig:sup-cyl-shock-theory}
\end{figure}

\begin{figure}[]
    \centering
    \begin{overpic}[width=0.95\textwidth,tics=5]{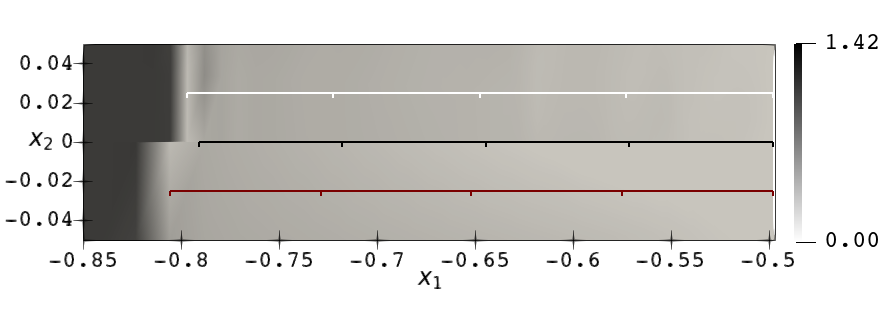}
    \definecolor{cnscolor}{HTML}{7a0000}
    \put (39,22) {\rotatebox{0}{\color{white}SSDC Eulerian $\Delta = 0.299$}}
    \put (39,16.5) {\rotatebox{0}{\color{black}Analytic $\Delta = 0.293$}}
    \put (39,11) {\rotatebox{0}{\color{cnscolor}SSDC CNS $\Delta = 0.308$}}
    \end{overpic}
    \vspace{0.3cm}
    \caption{Normalized distance between the shock and the
    cylinder leading edge for the flow past a cylinder at $\text{Re}=10,000$, and $\text{Ma} = 3.5$.
    The background is colored by the pointwise velocity magnitude.}
    %
    \label{fig:sup-cyl-distance}
\end{figure}


In this last test case, we present some results for the flow 
past a two-dimensional cylinder enclosed
between two solid walls \cite{boukharfane2018combined}. The similarity parameters are $\text{Re}=10,000$ and $\text{Ma} = 3.5$. 
The initial condition is a uniform 
flow with unit density, $\rho$, unit temperature, $T$, and only a non-zero 
unit velocity component in the $x_1$ direction, $\mathcal{U}_1$.
The flow is computed with a non-uniform distribution of the solution 
polynomial degree. Therefore, the solution is computed using the entropy stable $p$-nonconforming interface
technology \cite{FERNANDEZ2020104631}.
The solution polynomial degree distribution is shown in Figure \ref{fig:domain-p-super-cyl}, and
it is the same setup used in \cite{parsani2020high}.
Entropy stable adiabatic no-slip wall boundary conditions are enforced on the cylinder surface.
Inviscid (slip) wall boundary conditions are imposed on the top and bottom horizontal boundaries,
whereas far-field boundary conditions are used for the (left) inlet and the (right) outlet boundaries.
The mesh consists of $5,067$ quadrangles.
The solution computed with this setup is
denoted here as a ``numerically converged solution" in the sense that if we increase the order of
accuracy of the method, $p$, or we refine the mesh, the difference in the
solution for all the
primitive variables is machine precision.
Figure \ref{fig:sup-cyl-fields} shows the contour plot of the density, $\rho$, and the 
velocity component in the $x_1$ direction, $\mathcal{U}_1$, of the developed 
solution.

A quantitative analysis of the numerical results can be performed using the oblique shock wave 
theory \cite{shapiro1953dynamics}. 
At an oblique shock, the flow changes direction, and there are three directions of interest: the 
upstream and downstream flow directions and the shock wave direction itself.
The most useful relation of oblique shock wave theory is the one providing the deflection 
angle, $\theta$, as a function of the shock wave angle, $\beta$, and local Mach 
number, $\text{Ma}$ \cite{liepmann2001elements}:
\begin{equation}
    \tan (\theta)=2\cot (\beta)
    \left[\frac{\text{Ma}^2\sin^2(\beta)-1}
    {\text{Ma}^2(\gamma+\cos(2\beta))+2}\right].
\end{equation}
The pair $\theta-\beta$ obtained by postprocessing the solutions computed with the Eulerian and CNS models are shown in
Figure \ref{fig:sup-cyl-shock-theory}. We can see that the results for both models are in good agreement with the theoretical curve.
Finally, we report the normalized distance between the shock and the leading edge of the cylinder, as shown in Figure \ref{fig:sup-cyl-distance}. 
The value of the latter quantity is $\Delta=0.299$ and $\Delta=0.308$ 
for the Eulerian and CNS models, respectively. These values are in good agreement with the analytical value, which reads
$\Delta=0.293$ \cite{boiron2009high}.


\section{Conclusion}
\label{sec:conclusions}

Guided by the entropy stability analysis, we have derived solid wall boundary conditions 
that preserve the entropy stability of the Eulerian model for viscous compressible fluids 
proposed by \citet{svard2018new}. In that context, using summation-by-parts operators and the simultaneous-approximation
term technique, we have constructed entropy conservative and entropy stable solid wall boundary conditions
for the semidiscretized system, which mimics the continuous entropy analysis results.
The proposed boundary conditions have been validated in terms of accuracy,
entropy conservation, and entropy stability for a set of test cases of increasing complexity.
The numerical results obtained with the Eulerian model have been compared with the
solutions computed using the classic compressible Navier--Stokes equations also
discretized with an entropy stable methodology constructed using summation-by-parts
operators. Differences and similarities have then been highlighted. This study is a 
first attempt to understand better the effects of the viscous flux term introduced
in the conservation of mass equation.   


\par\addvspace{17pt}\small\rmfamily
\trivlist\if!{Data Availability}!\item[]\else
\item[\hskip\labelsep
{\bfseries{Data Availability}}]\fi
The datasets generated during and/or analysed during the current study are 
available in the Zenodo repository, \url{http://doi.org/10.5281/zenodo.5041436}
\endtrivlist\addvspace{6pt}

\begin{acknowledgements}
    The research reported in this paper was funded by King Abdullah 
    University of Science and Technology. We are thankful for the 
    computing resources of the Supercomputing Laboratory and the 
    Extreme Computing Research Center at King Abdullah University of 
    Science and Technology.
\end{acknowledgements}

%
%

\bibliographystyle{spbasic}      
\bibliography{ebc}   


\appendix

\section{Derivation of the time derivative of the entropy}\label{sec:derivation-of-time-derivative-of-entropy}

In this section, we provide the derivation of equation \eqref{eq:time-derivative-of-entropy}.
To obtain the expression of the time derivative of the entropy function,
$\mathbfcal{S}$, we multiply \eqref{eq:conservative-penalty}
and \eqref{eq:LDG-penalty} by $\mathbf{w}^\top\mathcal{P}$ and 
$\left(\left[\nu\frac{dq}{dw}\right]\Theta_{x_j}\right)^\top\mathcal{P}$,
respectively. This step yields
\begin{subequations}
    \begin{align}
        \begin{split}
            &\frac{d}{d t}\mathbf{1}^\top\widehat{\mathcal{P}}\mathbfcal{S}
            +\mathbf{1}^\top\widehat{\mathcal{P}}_{x_j,x_k}
            \widehat{\mathcal{B}}_{x_i}\mathbf{F}_{x_i}
            -\mathbf{w}^\top\mathcal{P}\mathcal{D}_{x_i}
            \left[\nu\frac{dq}{dw}\right]\mathbf{\Theta}_{x_j}\\
            &=\mathbf{w}^\top\mathcal{P}_{x_j,x_k}
            \left(\mathbf{g}^{(b),q}_{x_i}
            +\mathbf{g}^{(In),q}_{x_i}\right),
        \end{split}
        \label{eq:wt-conservative-penalty}\\
        \begin{split}
            &\left(\left[\nu\frac{dq}{dw}\right]\mathbf{\Theta}_{x_j}\right)^\top\mathcal{P}
            \mathbf{\Theta}_{x_i}
            -\left(\left[\nu\frac{dq}{dw}\right]\mathbf{\Theta}_{x_j}\right)^\top\mathcal{P}
            \mathcal{D}_{x_i}\mathbf{w}\\
            &=\left(\left[\nu\frac{dq}{dw}\right]\mathbf{\Theta}_{x_j}\right)^\top
            \mathcal{P}_{x_j,x_k}\left(\mathbf{g}^{(b),\Theta}_{x_i}
            +\mathbf{g}^{(In),\Theta}_{x_i}\right).
        \end{split}
        \label{eq:wt-LDG-penalty}
    \end{align}
\end{subequations}
Using the properties of SBP operators, we can write
\begin{equation}
    \mathbf{w}^\top\mathcal{P}\mathcal{D}_{x_i}
    \left[\nu\frac{dq}{dw}\right]\mathbf{\Theta}_{x_j}
    =\mathbf{w}^\top\mathcal{P}\mathcal{P}^{-1}_{x_i}\mathcal{B}_{x_i}
    \left[\nu\frac{dq}{dw}\right]\mathbf{\Theta}_{x_j}
    -\mathbf{w}^\top\mathcal{P}\mathcal{P}^{-1}_{x_i}\mathcal{Q}^\top_{x_i}
    \left[\nu\frac{dq}{dw}\right]\mathbf{\Theta}_{x_j}.
\end{equation}
Since $\mathcal{P}$ is diagonal, we have
\begin{equation}
    \mathcal{P}\mathcal{P}^{-1}_{x_i}\mathcal{Q}^\top_{x_i}
    =\mathcal{P}\mathcal{Q}^\top_{x_i}\mathcal{P}^{-1}_{x_i}
    =\mathcal{P}\mathcal{D}^\top_{x_i}
    =\mathcal{D}^\top_{x_i}\mathcal{P}.
\end{equation}
Thus, Equation \eqref{eq:wt-conservative-penalty} is recast in the following
form:
\begin{align}
    \label{eq:time-derivative-of-entropy-setup}
    \begin{split}
        &\frac{d}{d t}\mathbf{1}^\top\widehat{\mathcal{P}}\mathbfcal{S}
        +\mathbf{1}^\top\widehat{\mathcal{P}}_{x_j,x_k}
        \widehat{\mathcal{B}}_{x_i}\mathbf{F}_{x_i}
        +\left(\mathcal{D}_{x_i}\mathbf{w}\right)^\top\mathcal{P}
        \left[\nu\frac{dq}{dw}\right]
        \mathbf{\Theta}_{x_j}\\
        &=\mathbf{w}^\top\mathcal{P}_{x_j,x_k}\mathcal{B}_{x_i}
        \left[\nu\frac{dq}{dw}\right]
        \mathbf{\Theta}_{x_j}
        +\mathbf{w}^\top\mathcal{P}_{x_j,x_k}
        \left(\mathbf{g}^{(b),q}_{x_i}
        +\mathbf{g}^{(In),q}_{x_i}\right).
    \end{split}
\end{align}
Additionally, because $\left[\nu\frac{dq}{dw}\right]$
is SPD, we have
\begin{equation}
    \left(\left[\nu\frac{dq}{dw}\right]\mathbf{\Theta}_{x_j}\right)^\top\mathcal{P}
    \mathbf{\Theta}_{x_i}=
    \left(\left[\nu\frac{dq}{dw}\right]^{\frac12}\mathbf{\Theta}_{x_j}\right)^\top\mathcal{P}
    \left(\left[\nu\frac{dq}{dw}\right]^{\frac12}\mathbf{\Theta}_{x_i}\right)
    =\left\|\left[\nu\frac{dq}{dw}\right]^{\frac12}\mathbf{\Theta}_{x_i}\right\|^2_\mathcal{P}.
\end{equation}
Now, substituting \eqref{eq:wt-LDG-penalty} into \eqref{eq:time-derivative-of-entropy-setup}
gives the following expression
\begin{align}
\label{eq:tobesimple}
    \begin{split}
        &\frac{d}{d t}\mathbf{1}^\top\widehat{\mathcal{P}}\mathbfcal{S}
        +\left\|\left[\nu\frac{dq}{dw}\right]^{\frac12}\mathbf{\Theta}_{x_i}\right\|^2_\mathcal{P}\\
        &=-\mathbf{1}^\top\widehat{\mathcal{P}}_{x_j,x_k}
        \widehat{\mathcal{B}}_{x_i}\mathbf{F}_{x_i}\\
        &+\mathbf{w}^\top\mathcal{P}_{x_j,x_k}\mathcal{B}_{x_i}
        \left[\nu\frac{dq}{dw}\right]\mathbf{\Theta}_{x_j}\\
        &+\mathbf{w}^\top\mathcal{P}_{x_j,x_k}
        \left(\mathbf{g}^{(b),q}_{x_i}
        +\mathbf{g}^{(In),q}_{x_i}\right)\\
        &+\left(\left[\nu\frac{dq}{dw}\right]\mathbf{\Theta}_{x_j}\right)^\top
        \mathcal{P}_{x_j,x_k}\left(\mathbf{g}^{(b),\Theta}_{x_i}
        +\mathbf{g}^{(In),\Theta}_{x_i}\right).
    \end{split}
\end{align}
By defining
\begin{equation}
    \mathbf{DT}=\left\|\left[\nu\frac{dq}{dw}\right]^{\frac12}\mathbf{\Theta}_{x_i}\right\|^2_\mathcal{P},
\end{equation}
and
\begin{align}
    \begin{split}
        \Xi&=-\mathbf{1}^\top\widehat{\mathcal{P}}_{x_j,x_k}
        \widehat{\mathcal{B}}_{x_i}\mathbf{F}_{x_i}\\
        &+\mathbf{w}^\top\mathcal{P}_{x_j,x_k}\mathcal{B}_{x_i}
        \left[\nu\frac{dq}{dw}\right]\mathbf{\Theta}_{x_j}\\
        &+\mathbf{w}^\top\mathcal{P}_{x_j,x_k}
        \left(\mathbf{g}^{(b),q}_{x_i}
        +\mathbf{g}^{(In),q}_{x_i}\right)\\
        &+\left(\left[\nu\frac{dq}{dw}\right]\Theta_{x_j}\right)^\top
        \mathcal{P}_{x_j,x_k}\left(\mathbf{g}^{(b),\Theta}_{x_i}
        +\mathbf{g}^{(In),\Theta}_{x_i}\right),
    \end{split}
\end{align}
Equation \eqref{eq:tobesimple} can be further simplified to obtain
\eqref{eq:time-derivative-of-entropy}, \textit{i.e.},
\begin{align}
    &\frac{d}{d t}\mathbf{1}^\top\widehat{\mathcal{P}}\mathbfcal{S}
    +\mathbf{DT}=\Xi.
\end{align}


\section{Change of variables matrix, $\frac{dq}{dw}(v)$}
\label{sec:dqdw}

The Jacobian of the conservative variables, $q$, 
in terms of the entropy variables, $w$, as a function of the primitive
variables, $v$, is given by
\begin{equation}
    \frac{dq}{dw}(v)=
    \begin{bmatrix}
        \frac{\rho }{R} 
        & \frac{\mathcal{U}_1 \rho }{R} 
        & \frac{\mathcal{U}_2 \rho }{R} 
        & \frac{\mathcal{U}_3 \rho }{R} 
        & \frac{dq}{dw}(v)_{1,5}
        \\
        \frac{\mathcal{U}_1 \rho }{R} 
        & \left(\mathcal{T}+\frac{\mathcal{U}_1^2}{R}\right) \rho  
        & \frac{\mathcal{U}_1 \mathcal{U}_2 \rho }{R} 
        & \frac{\mathcal{U}_1 \mathcal{U}_3 \rho }{R} 
        & \frac{dq}{dw}(v)_{2,5}
        \\
        \frac{\mathcal{U}_2 \rho }{R} 
        & \frac{\mathcal{U}_1 \mathcal{U}_2 \rho }{R} 
        & \left(\mathcal{T}+\frac{\mathcal{U}_2^2}{R}\right) \rho  
        & \frac{\mathcal{U}_2 \mathcal{U}_3 \rho }{R} 
        & \frac{dq}{dw}(v)_{3,5}
        \\
        \frac{\mathcal{U}_3 \rho }{R} 
        & \frac{\mathcal{U}_1 \mathcal{U}_3 \rho }{R} 
        & \frac{\mathcal{U}_2 \mathcal{U}_3 \rho }{R} 
        & \left(\mathcal{T}+\frac{\mathcal{U}_3^2}{R}\right) \rho  
        & \frac{dq}{dw}(v)_{4,5}
        \\
        \frac{dq}{dw}(v)_{1,5} 
        & \frac{dq}{dw}(v)_{2,5} 
        & \frac{dq}{dw}(v)_{3,5} 
        & \frac{dq}{dw}(v)_{4,5}
        & \frac{dq}{dw}(v)_{5,5}
    \end{bmatrix},
\end{equation}
where
\begin{align}
    &\frac{dq}{dw}(v)_{1,5}
    =\frac{\rho\left(-2 R\mathcal{T}+\|\mathcal{U}\|_{L^2}^2+2\mathcal{T}c_{pr}\right)}{2 R},
\end{align}

\begin{align}
    &\frac{dq}{dw}(v)_{2,5}
    =\frac{\mathcal{U}_1 \rho  \left(\|\mathcal{U}\|_{L^2}^2+2\mathcal{T}c_{pr}\right)}{2R},
\end{align}

\begin{align}
    &\frac{dq}{dw}(v)_{3,5}
    =\frac{\mathcal{U}_2 \rho  \left(\|\mathcal{U}\|_{L^2}^2+2\mathcal{T}c_{pr}\right)}{2R}, 
\end{align}

\begin{align}
    &\frac{dq}{dw}(v)_{4,5}
    =\frac{\mathcal{U}_3 \rho  \left(\|\mathcal{U}\|_{L^2}^2+2\mathcal{T}c_{pr}\right)}{2R}, 
\end{align}

\begin{align}
    &\frac{dq}{dw}(v)_{5,5}
    =\frac{\rho  \left(\left(\|\mathcal{U}\|_{L^2}^2\right)^2
    +4\mathcal{T} c_{pr} \left(-R\mathcal{T}+\|\mathcal{U}\|_{L^2}^2+\mathcal{T} c_{pr}\right)\right)}{4 R}.
\end{align}


\section{Viscous boundary penalty}\label{sec:viscous-boundary-penalty}

This section constructs the viscous boundary penalty term for an adiabatic wall and a wall with heat entropy flux. We then show that 
the resulting expressions lead to an entropy conservative and entropy stable
penalty terms, respectively. 

Assuming a dissipative internal penalty term, $\mathcal{M}=0$, the RHS of 
\eqref{eq:time-derivative-of-entropy} becomes
\begin{equation}\label{eq:viscous-boundary-penalty-setup}
    \frac{d}{d t}\mathbf{1}^\top\widehat{\mathcal{P}}\mathbfcal{S}
    +\mathbf{DT} 
    =w^\top\widehat{\mathcal{P}}_{x_j,x_k}
    \left(g^{(b,V),q}_{x_i}\right)
    +\Theta_{x_i}^\top\left(\nu\frac{dq}{dw}\right)
    \widehat{\mathcal{P}}_{x_j,x_k}g^{(b),\Theta}_{x_i}.
\end{equation}
Define $\mathbf{W}=\text{diag}(\mathbf{w})$ and write
\begin{equation}\label{eq:1-transpose-P}
   \mathbf{w}^\top\widehat{\mathcal{P}}_{x_j,x_k}\mathbf{\chi}
   =\mathbf{1}^\top\mathbf{W}\widehat{\mathcal{P}}_{x_j,x_k}\mathbf{\chi}
   =\mathbf{1}^\top\widehat{\mathcal{P}}_{x_j,x_k}\mathbf{W}\mathbf{\chi},
\end{equation}
where $\mathbf{W}\mathcal{P}_{x_j,x_k}=\mathcal{P}_{x_j,x_k}\mathbf{W}$, since $\mathbf{W}$ is diagonal,
and $\mathbf{\chi}$ is a vector on a node.
We use \eqref{eq:1-transpose-P} for both terms in 
the RHS of \eqref{eq:viscous-boundary-penalty-setup}. The first term reads
\begin{equation}\label{eq:conservative-viscous-penalty}
    \frac12\mathbf{1}^\top\widehat{\mathcal{P}}_{x_j,x_k}
    \widehat{\mathcal{B}}^-_{x_i}\mathbf{W}
    \left(\nu\frac{d q}{d w}\right)
    \Theta_{x_i}
    -\frac12\mathbf{1}^\top\widehat{\mathcal{P}}_{x_j,x_k}
    \widehat{\mathcal{B}}^-_{x_i}\mathbf{W}
    \left(\nu\frac{d q}{d w}\left(v^{(b,V)}\right)\right)
    \widetilde{\Theta}_{x_i},
\end{equation}
where both $\frac{d q}{d w}\left(v^{(b,V)}\right)$ and 
$\widetilde{\Theta}$ are evaluated using the manufactured primitive state
\begin{equation}
    v^{(b,V)}=\left(\rho,
    -\mathcal{U}_1+2\mathcal{U}_1^\text{wall},
    -\mathcal{U}_2+2\mathcal{U}_2^\text{wall},
    -\mathcal{U}_3+2\mathcal{U}_3^\text{wall},\mathcal{T}\right)^\top.
\end{equation}
The kinematic viscosity, $\nu$, depends only on the density $\rho$.
Thus, its evaluation is the same for $v$ and $v^{(b,V)}$.
The procedure for computing $\widetilde{\Theta}_{x_i}$ 
is detailed in \cite{dalcin2019conservative} and is summarized
in the following equation. The manufactured entropy variables gradient
is defined as
\begin{equation}
    \widetilde{\Theta}_{x_i}
    =\frac{dw}{dv}\left(v^{(b,V)}\right)\tilde{\Pi}
    =\frac{dw}{dv}\left(v^{(b,V)}\right)\text{diag}(-1,1,1,1,-1)\Pi,
\end{equation}
where $\Pi$ is the gradient of primitive variables.
The second term on the RHS of \eqref{eq:viscous-boundary-penalty-setup} 
is written as
\begin{equation}\label{eq:entropy-viscous-penalty}
    \frac12\mathbf{1}^\top\widehat{\mathcal{P}}_{x_j,x_k}
    \widehat{\mathcal{B}}^-_{x_i}
    \mathbf{W}
    \left(\nu\frac{d q}{d w}\right)
    \Theta_{x_i}
    -\frac12\mathbf{1}^\top\widehat{\mathcal{P}}_{x_j,x_k}
    \widehat{\mathcal{B}}^-_{x_i}
    \mathbf{W}\left(v^{(b,V)}\right)
    \left(\nu\frac{d q}{d w}\right)
    \Theta_{x_i}.
\end{equation}
The second terms in \eqref{eq:conservative-viscous-penalty} and 
\eqref{eq:entropy-viscous-penalty} cancel out.
On the other hand, the first terms add up to
\begin{equation}
    \mathbf{1}^\top\widehat{\mathcal{P}}_{x_j,x_k}
    \widehat{\mathcal{B}}^-_{x_i}\mathbf{W}
    \left(\nu\frac{d q}{d w}\right)
    \Theta_{x_i}
    =\mathbf{1}^\top\widehat{\mathcal{P}}_{x_j,x_k}g(t).
\end{equation}
Therefore, we arrive at \eqref{eq:viscous-boundary-penalty}, \textit{i.e.},
\begin{equation}
    \frac{d}{d t}\mathbf{1}^\top\widehat{\mathcal{P}}\mathcal{S}+DT
    =\mathbf{1}^\top\widehat{\mathcal{P}}_{x_j,x_k}g(t).
\end{equation}
For an adiabatic wall $g(t)=0$, and thus, the RHS is zero yielding an entropy 
conservative term. For a wall with non-zero heat entropy flux, $g(t)$ is data, and the RHS is bounded.


\section{Interior penalty term, $\mathcal{M}^{(b,V)}$}\label{sec:dissipative-internal-penalty-term}

In this section, we compute the contribution to time derivative of
the entropy function, $\mathcal{S}$, from the interior penalty term $\mathcal{M}^{(b,V)}$. 
Assuming a uniform hexahedral element (\textit{i.e.}, a cube) of length one and a wall in the $x_i$ direction,
we use the manufactured state of the vector of primitive variables at the wall
\eqref{eq:viscous-boundary-primitive}
to construct the jump in the entropy variables in the normal direction
to the wall, $\hat{n}$,
\begin{equation}
    \Delta w^{(b,V)}=(w-w^{(b,V)})\cdot \hat{n}^\top, 
\end{equation}
and the viscous boundary flux state
\begin{align}
    f^{(V)}(v^{(b,V)})
    =\left(\nu\frac{d q}{d w}(v^{(b,V)})\right)
    \Delta w^{(b,V)}\cdot \hat{n}.
\end{align}
Then, we use an average of the viscous flux of the numerical solution 
state at the wall and the manufactured state
\begin{equation}
    \mathcal{M}^{(b,V)}=-\beta\frac{f^{(V)}(v)+f^{(V)}(v^{(b,V)})}{2},
\end{equation}
to compute the dissipative term with Mathematica \cite{Mathematica}
\begin{equation}
    w^\top\mathcal{M}^{(b,V)}
    =-\frac{2\beta \alpha \mu}{R\mathcal{T}^2}\|\Delta \mathcal{U}\|^2
    \left(\|\Delta \mathcal{U}\|^2+R\mathcal{T}\right),
\end{equation}
where $\Delta \mathcal{U}$ is the jump in the velocity of the numerical solution 
state at the wall and the manufactured state. The term is negative for any non-zero jump in velocity, 
and thus, is entropy dissipative.



\end{document}